\documentclass[onefignum,onetabnum]{siamart190516}

\usepackage{mathtools}%

\usepackage{lipsum}
\usepackage{amsfonts}
\usepackage{graphicx}
\usepackage{epstopdf}
\usepackage{algorithmic}
\ifpdf
  \DeclareGraphicsExtensions{.eps,.pdf,.png,.jpg}
\else
  \DeclareGraphicsExtensions{.eps}
\fi

\newsiamremark{remark}{Remark}
\newsiamremark{hypothesis}{Hypothesis}
\crefname{hypothesis}{Hypothesis}{Hypotheses}
\newsiamthm{claim}{Claim}

\headers{Projected subgradient method}{M. Louzeiro, C. Kawan, S. Hafstein, P. Giesl, and J.~Yuan}

\title{A projected subgradient method for the computation of adapted metrics for dynamical systems\thanks{Submitted to the editors DATE.
\funding{This work was funded by the German Research Foundation (DFG) through the grant ZA 873/4-1 and by through the grant NSFC 12171087.}}}

\author{Mauricio Louzeiro\thanks{Dongguan University of Technology, School of Computer Science and Technology, Dongguan, Guangdong, China (\email{mauriciosilvalouzeiro@hotmail.com})} \and Christoph Kawan\thanks{Institute of Informatics, LMU Munich, Oettingenstra{\ss}e 67, 80538 M\"{u}nchen, Germany (\email{christoph.kawan@lmu.de}).}
\and Sigurdur Hafstein\thanks{The Science Institute, University of Iceland, Dunhagi 5, 107 Reykjavik, Iceland
  (\email{shafstein@hi.is}).}
\and Peter Giesl\thanks{Department of Mathematics, University of Sussex, Falmer, BN1 9QH, United Kingdom (\email{P.A.Giesl@sussex.ac.uk}).}\and Jinyun Yuan\thanks{Dongguan University of Technology, School of Computer Science and Technology, Dongguan, Guangdong, China (\email{yuanjy@gmail.com})}.
  }

\usepackage{amsopn}

\newcommand{\R}{\mathbb{R}}%
\newcommand{\N}{\mathbb{N}}%
\newcommand{\tm}{\times}%
\newcommand{\tr}{\mathrm{tr}}%
\newcommand{\GL}{\mathrm{GL}}%
\newcommand{\SO}{\mathrm{SO}}%
\newcommand{\trn}{^{\scriptscriptstyle \top}}%
\newcommand{\rme}{\mathrm{e}}%
\newcommand{\rmD}{\mathrm{D}}%
\newcommand{\rmd}{\mathrm{d}}%
\newcommand{\fa}{\mathfrak{a}}%
\newcommand{\inner}{\mathrm{int}}%
\newcommand{\cl}{\mathrm{cl}}%
\newcommand{\res}{\mathrm{res}}%

\newcommand{\MC}{\mathcal{M}}%
\newcommand{\NC}{\mathcal{N}}%
\newcommand{\SC}{\mathcal{S}}%
\newcommand{\XC}{\mathcal{X}}%

\newcommand{\dom}{\mathrm{dom}}%

\newcommand{\PC}{\mathcal{P}}%
\newcommand{\diam}{\operatorname{diam}}%
\newcommand{\diag}{\operatorname{diag}}%

\begin{document}

\maketitle

\begin{abstract}
In this paper, we extend a recently established subgradient method for the computation of Riemannian metrics that optimizes certain singular value functions associated with dynamical systems. This extension is threefold. First, we introduce a projected subgradient method which results in Riemannian metrics whose parameters are confined to a compact convex set and we can thus prove that a minimizer exists; second, we allow inexact subgradients and study the effect of the errors on the computed metrics; and third, we analyze the subgradient algorithm for three different choices of step sizes: constant, exogenous and Polyak. The new methods are illustrated by application to dimension and entropy estimation of the H\'enon map.
\end{abstract}



\section{Introduction}
Some of the central problems in the theory of smooth dynamical systems are related to the optimization of certain singular value functions of the derivative cocycle associated with the system. These problems include the computation of Lyapunov exponents, the computation of contraction metrics for exponentially stable equilibria and the upper estimation of the  dimension of attractors. A more recent objective, motivated by problems in information-based control, is the computation of the minimal data rate at which sensory data has to be transmitted to an observer at a remote location such that it can produce reliable state estimates. Given a discrete-time system induced by a smooth map $\phi:\R^n \rightarrow \R^n$, in most of these problems one is interested in the minimization of a function of the type%
\begin{equation*}
  \Sigma_k:P \mapsto \max_{x \in K}\sum_{i=1}^k \log \alpha_i^P(x),%
\end{equation*}
where $K$ is a compact forward-invariant set of $\phi$ and $\alpha_1^P(x) \geq \alpha_2^P(x) \geq \ldots \geq \alpha_n^P(x)$ denote the singular values of the derivative $\rmD\phi(x)$, computed with respect to a Riemannian metric $P(\cdot)$ on $K$ (see Subsection \ref{subsec_problem} for the technical definition). For instance, in the case of contraction analysis for equilibria, one is interested in finding a metric $P(\cdot)$ such that $\Sigma_1(P) < 1$, and in the case of the remote estimation problem, the smallest data rate is given by the minimum of $P \mapsto \max \{ \Sigma_k(P) : 0 \leq k \leq n \}$.%

In \cite{kawan2021subgradient}, we established a convexity property of the functions $\Sigma_k$, which allowed us to solve a restricted version of the minimization problem via the subgradient algorithm on manifolds with lower bounded sectional curvature \cite{FerreiraLouzeiroPrudente2019:subgrad,Wang2018,Ferreira2006}. Namely, we had to restrict the class of Riemannian metrics to metrics of the form $P(x) = \rme^{r(x)}p$ with a polynomial $r(x)$ and a positive definite matrix $p$. The problem can then be formulated as a geodesically convex optimization problem on a manifold of the form $\R^N \tm \SC^+_n$, where the Euclidean component $\R^N$ serves as the parameter space for the polynomials and $\SC^+_n$ is the Riemannian manifold of positive definite matrices. We demonstrated the efficiency of the algorithm by applying it to the computation of the minimal data rate in the remote state estimation problem (described by the notion of \emph{restoration entropy}) for three nonlinear systems - the H\'{e}non system, a bouncing ball system and the Lorenz system. In all three examples, we observed a fast convergence and, where the exact value of the entropy was known, the same value was obtained by our algorithm with a high accuracy.%

Three drawbacks of the algorithm developed in \cite{kawan2021subgradient} are the following: first, the computation of a subgradient, which has to be performed in each step of the algorithm, involves a maximization problem that can, in general, only be solved by brute force. This slows  the algorithm down and additionally introduces an error in the computed subgradient that is hard to quantify. Second, we do not know general conditions under which the optimization problem admits a solution. This implies that the algorithm does not necessarily converge. In particular, if no minimizer exists, the computed metric could become asymptotically singular. Finally, using the exogenous step size rule, we cannot provide an explicit estimate on the rate of convergence even if a minimizer exists. In this paper, we attempt to overcome all these drawbacks to a certain extent. %

A version of the classical (Euclidean) subgradient algorithm is the projected subgradient algorithm in which the search for a minimizer is only carried out within a compact convex subset of the given domain. The compactness guarantees that for this restricted problem a minimizer exists. Moreover, any upper bound on the diameter of the compact set together with an estimate on the maximal norm of the subgradient allows us to provide explicit estimates on the number of steps necessary to achieve a given accuracy. In this paper, we analyze an inexact projected subgradient algorithm on Hadamard manifolds and provide convergence results for three different choices of step sizes:  exogenous, constant and Polyak. In particular, we prove that errors in the computation of the subgradient do not accumulate as the number of steps grows.%

We then study order intervals on $\SC^+_n$ as candidates for compact convex subsets, which can be used in the projected subgradient algorithm for singular value optimization. We further analyze the error in the computation of the subgradient for singular value optimization, which turns out to depend sensitively on a spectral gap in the singular value spectrum. Finally, we compute a Lipschitz constant for the functions $\Sigma_k$, which only depends on the dimension $n$. This analysis allows us to apply the subgradient algorithm with different types of step sizes to the problems of dimension and data rate estimation for the H\'enon system.%

The advantages of the projected subgradient algorithm thus are the following: the existence of a minimizer is guaranteed and we can prove the convergence; we can ensure that the computed Riemannian metrics do not have values that are too close to the boundary of $\SC^+_n$, which is important in practical applications; we can provide explicit estimates on the rate of convergence and we can use Polyak step sizes, which delivered the best upper bound on the restoration entropy for our example system, see Section \ref{example:resent}.

Let us give an overview of the paper: In Section \ref{sec:2} we introduce notations as well as the subgradient algorithm. In Section \ref{sec:3} we introduce the projected subgradient method on Hadamard manifolds with inexact subgradients and perform a convergence analysis with three choices of step sizes. Section \ref{sec:4} applies the projected inexact subgradient method to the singular value optimization problems and presents further results, in particular an error analysis. In Secion \ref{sec:5} the methods are applied to derive upper bounds on the dimension and entropy of the H\'enon system, before conclusions in Section \ref{sec:5}. The appendix proves a bound the sectional curvature of $\SC_n^+$.

\section{Preliminaries}\label{sec:2}

\subsection{General notation and definitions}

We write $\N = \{1,2,3,\ldots\}$ for the natural numbers and $\N_0 := \N \cup \{0\}$. By $\R$ and $\R_+$, we denote the reals and non-negative reals, respectively. By $\ln$ ($\log_2$), we denote the natural logarithm (logarithm to the base $2$).%

\paragraph{Matrices} We write $\R^{n \tm n}$ for the vector space of all $n \tm n$ real matrices. The notations $A\trn$, $\det(A)$ and $\tr(A)$ stand for the transpose, the determinant and the trace of $A \in \R^{n \tm n}$, respectively. For the singular values of $A$, i.e.~the eigenvalues of $\sqrt{A\trn A}$, we write $\alpha_1(A) \geq \ldots \geq \alpha_n(A)$. By $I_{n \tm n}$ and $0_{n \tm n}$, we denote the identity and the zero matrix in $\R^{n \tm n}$, respectively. If the dimension is clear from the context, we also omit the subscript. The general linear group is denoted by $\GL(n,\R) = \{ A \in \R^{n \tm n} : A^{-1} \mbox{ exists} \}$. The notation $\diag(\lambda_1,\ldots,\lambda_n)$ is used for the diagonal matrix with entries $\lambda_1,\ldots,\lambda_n$. We write $\langle \cdot,\cdot \rangle_F$ for the standard inner product on $\R^{n \tm n}$ which is given by $\langle X,Y \rangle_F = \tr(XY\trn)$ and induces the Frobenius norm $\|X\|_F = \sqrt{\tr(XX\trn)}$.%

\paragraph{Manifolds} Next, we present some notations, definitions and basic properties of Riemannian manifolds used throughout the paper; for more details, we refer the reader to \cite{DoCa92, Sakai1996, lee2006riemannian, Bacak:2014:2}.%

If $\MC$ is a smooth manifold, we write $T_p\MC$ for the tangent space at $p \in \MC$. The derivative of a smooth mapping $f:\MC \rightarrow \NC$ between manifolds at $p \in \MC$ is denoted by $\rmD f(p):T_p\MC \rightarrow T_{f(p)}\NC$. If $K \subset \MC$ is any subset, we write $C^k(K,\NC)$ for the set of $k$ times continuously differentiable maps from $K$ into $\NC$, where $k = 0$ means continuous. We write $\inner(A)$ and $\cl(A)$ for the interior and closure of a subset $A \subset \MC$, respectively. If $\MC$ is equipped with a Riemannian metric, we write $\langle \cdot,\cdot \rangle_p$ and $\|\cdot\|_p$ for the associated inner product and norm on $T_p\MC$, respectively. If the point $p$ is clear from the context, we also omit the subscript $p$ in the inner product and the norm. By $\exp_p$, we denote the Riemannian exponential map at $p \in M$. The diameter of a subset $\XC \subset \MC$ is given by $\diam(\XC) = \sup_{x,y\in\XC}d(x,y)$, where $d(\cdot,\cdot)$ is the Riemannian distance function on $\MC$. The sectional curvature of the two-dimensional subspace of $T_p\MC$, spanned by two linearly independent vectors $x,y \in T_p\MC$, is defined by%
\begin{equation}\label{eq:secc-curv}
   K(x,y) := \frac{ \langle R(x,y)y,x \rangle}{\|x\|^2\|y\|^2 - \langle x,y \rangle^2},%
\end{equation}
where $R$ is the curvature tensor of $\MC$, evaluated at the point $p$.%

A Riemannian manifold $\MC$ is called a \emph{Hadamard manifold} if it is simply connected, complete and its sectional curvature is non-positive everywhere. In this case, for each $p \in \MC$, the exponential map $\exp_p \colon T_p\MC \to \MC$ is a diffeomorphism, which implies that the geodesics connecting any two distinct points exist and are unique. We write $\log_p:\MC \rightarrow T_p\MC$ for the inverse of $\exp_p$.%

The geodesic $\gamma \colon [0,1] \to \MC$ connecting $p$ and $q$ is denoted by $\gamma_{p,q}$. A nonempty subset $\XC \subseteq \MC$ is called \emph{(geodesically) convex} if the image of $\gamma_{p,q}$ lies completely in $\XC$ for any two points $p,q \in \XC$. If $\XC$ is also closed, then the mapping $\PC_{\XC}(p) \coloneqq \{q\in \MC \colon d(p,q) = \inf_{q'\in \XC}d(p,q') \}$ is single-valued and non-expansive, i.e.%
\begin{equation}\label{eq:nonexp-proj}
	d(\PC_{\XC}(p),\PC_{\XC}(q)) \leq d(p,q) \mbox{\quad for all\ } p,q \in \MC.%
\end{equation}
In this case, $\PC_{\XC}(p)$ is called the \emph{projection} of $p$ onto $\XC$.%

If $\MC$ is a Hadamard manifold, a function $f\colon \MC \to \overline{\R} := \R \cup \{\pm \infty\}$ is called%
\begin{enumerate}
\item[(i)] \emph{proper} if $\dom f \coloneqq \{ p \in \MC \colon f(p) < \infty \} \neq \emptyset$ and $f(p) > -\infty$ for all $p \in \MC$.%
\item[(ii)]	\emph{convex} if, for all $p,q \in \MC$, the composition $f \circ \gamma_{p,q} \colon [0,1] \subset \R \to \overline{\R}$ is a convex function on $[0,1]$ in the classical sense.%
\item[(iii)]\label{item:lsc} \emph{lower semicontinuous} at $p\in \dom f$ if for each sequence $\{p_k\}$ converging to $p$, we have $\liminf_{k \to \infty} f(p_k)\geq f(p)$.%
\end{enumerate}

The \emph{subdifferential} of a proper convex function $f:\MC \rightarrow \overline{\R}$ at $p\in\dom f$ is defined by%
\begin{equation}\label{eq:defSugrad}
	\partial f(p) := \left\{s \in T_p\MC \colon f(q) \geq f(p)+\left\langle s,\log_pq\right\rangle,\ \forall q\in\dom f\right\}.%
\end{equation}
Each element of $\partial f(p)$ is called a \emph{subgradient} of $f$ at $p$. We remark that $\partial f(p)$ is nonempty for all $p \in \inner(\dom f)$; see \cite[Prop.~2.5]{WangLiWangYao2015}.%

\paragraph{Positive definite matrices} Next, we recall some facts about the space of positive definite matrices. For a comprehensive treatment, we refer to \cite[Ch.~6]{bhatia2009positive}.%

We define $\SC_n := \{ A \in \R^{n \tm n} : A = A\trn \}$ and write $A \leq B$ ($A < B$) if $A,B \in \SC_n$ and $B - A$ is positive semidefinite (positive definite). The eigenvalues of any $A \in \SC_n$ are denoted by $\lambda_1(A) \geq \ldots \geq \lambda_n(A)$. The space of all positive definite matrices is denoted by $\SC^+_n := \{ A \in \SC_n : A > 0 \}$. As an open subset of the Euclidean space $\SC_n$, it is a smooth manifold. The tangent space of $\SC^+_n$ at any point $p$ can be canonically identified with $\SC_n$. We equip $\SC^+_n$ with the so-called \emph{trace metric}%
\begin{equation*}
  \langle v,w \rangle_p := \tr(p^{-1}vp^{-1}w) \mbox{\quad for all\ } p \in \SC^+_n,\ v,w \in T_p\SC^+_n = \SC_n.%
\end{equation*}
Some facts about the Riemannian manifold $\SC^+_n$, used in this paper, are the following:%
\begin{itemize}
\item $\SC^+_n$ is a Hadamard manifold as well as a symmetric space.%
\item The group $\GL(n,\R)$ acts transitively on $\SC^+_n$ by isometries. This action is given by $(g,p) \mapsto gpg\trn$.%
\item The inversion $g \mapsto g^{-1}$ is an isometry of $\SC^+_n$.%
\item The unique geodesic joining two points $p,q \in \SC^+_n$ is given by%
\begin{equation*}
  \gamma_{p,q}(\theta) = p \#_{\theta}\, q := p^{\frac{1}{2}}(p^{-\frac{1}{2}}qp^{-\frac{1}{2}})^{\theta} p^{\frac{1}{2}},\quad \theta \in [0,1].%
\end{equation*}
We also write $p \#\, q := p \#_{\frac{1}{2}}\, q$ for the midpoint of the geodesic.%
\item The unique geodesic $\gamma$ with $\gamma(0) = p$ and $\dot{\gamma}(0) = v$ is given by%
\begin{equation*}
  \gamma(\theta) = p^{\frac{1}{2}}\exp(\theta p^{-\frac{1}{2}}vp^{-\frac{1}{2}})p^{\frac{1}{2}},\quad \theta \in \R.%
\end{equation*}
\item The distance function on $\SC^+_n$ can be written as%
\begin{equation}\label{eq_snp_dist}
  d(p,q) = \Bigl(\sum_{i=1}^n \ln^2 \lambda_i(p^{-1}q)\Bigr)^{\frac{1}{2}}.%
\end{equation}
Observe that although $p^{-1}q$ is not necessarily symmetric, it has real eigenvalues, since it is similar to the symmetric matrix $p^{-1/2}qp^{-1/2}$.%
\end{itemize}

\subsection{The subgradient method for singular value optimization}\label{subsec_problem}

For any $g \in \GL(n,\R)$, we write%
\begin{equation*}
  \vec{\sigma}(g) := (\log_2\alpha_1(g),\ldots,\log_2\alpha_n(g)),%
\end{equation*}
which defines a mapping from $\GL(n,\R)$ into the cone%
\begin{equation*}
  \fa^+ := \{ \xi \in \R^n: \xi_1 \geq \xi_2 \geq \ldots \geq \xi_n \}.%
\end{equation*}
On $\fa^+$, we define the partial order%
\begin{equation*}
  \xi \preceq \eta :\Leftrightarrow \left\{\begin{array}{cl} \xi_1 + \cdots + \xi_k \leq \eta_1 + \cdots + \eta_k & \mbox{for } k = 1,\ldots,n-1,\\
	             \xi_1 + \cdots + \xi_k = \eta_1 + \cdots + \eta_k & \mbox{for } k = n.%
						\end{array}\right.%
\end{equation*}

Now, consider a discrete-time dynamical system%
\begin{equation*}
  x_{t+1} = \phi(x_t)%
\end{equation*}
with a $C^1$-map $\phi:\R^n \rightarrow \R^n$. We assume that $K \subset \R^n$ is a compact forward-invariant set, i.e.~$\phi(K) \subseteq K$. Moreover, we assume that $K = \cl(\inner K)$ and%
\begin{equation*}
  A(x) := \rmD\phi(x) \in \GL(n,\R) \mbox{\quad for all\ } x \in K.%
\end{equation*}
A Riemannian metric on $K$ can be regarded as a continuous map $P:K \rightarrow \SC^+_n$. We define the singular values $\alpha_1^P(x) \geq \ldots \geq \alpha_n^P(x)$ of $A(x)$ with respect to the metric $P$ as the eigenvalues of $[B(x)B(x)\trn]^{\frac{1}{2}}$, where%
\begin{equation*}
  B(x) := P(\phi(x))^{\frac{1}{2}} A(x) P(x)^{-\frac{1}{2}}.%
\end{equation*}
That is, $\alpha_i^P(x)$ are the ordinary singular values of $B(x)$, or the singular values of $A(x)$ regarded as a linear operator between the inner product spaces $(\R^n,\langle P(x) \cdot,\cdot \rangle)$ and $(\R^n,\langle P(\phi(x))\cdot,\cdot \rangle)$, respectively, see \cite[Lem.~5]{kawan2021remote}.%

In \cite[Lem.~3.2]{kawan2021subgradient}, we have proved that for any two metrics $P$ and $Q$, the relation%
\begin{align}\label{eq_vf_convexity}
\begin{split}
  &\vec{\sigma}([P(\phi(x)) \#_{\theta}\, Q(\phi(x))]^{\frac{1}{2}}A(x)[P(x) \#_{\theta}\, Q(x)]^{-\frac{1}{2}}) \\
	&\preceq (1-\theta)\vec{\sigma}(P(\phi(x))^{\frac{1}{2}}A(x)P(x)^{-\frac{1}{2}}) + \theta \vec{\sigma}(Q(\phi(x))^{\frac{1}{2}}A(x)Q(x)^{-\frac{1}{2}})%
\end{split}
\end{align}
holds for all $\theta \in [0,1]$ and $x \in K$. This can be regarded as a form of (geodesic) convexity. For each $k \in \{1,\ldots,n\}$ and $x \in K$, we introduce the function%
\begin{equation*}
  \Sigma_{k,x}:C^0(K,\SC^+_n) \rightarrow \R,\quad \Sigma_{k,x}(P) := \sum_{i=1}^k \log_2 \alpha_i^P(x).%
\end{equation*}
Then, \eqref{eq_vf_convexity} implies that for all $P,Q \in C^0(K,\SC^+_n)$ and $\theta \in [0,1]$ the inequality%
\begin{equation*}
  \Sigma_{k,x}(P \#_{\theta}\, Q) \leq (1 - \theta)\Sigma_{k,x}(P) + \theta\Sigma_{k,x}(Q)%
\end{equation*}
holds, where $P \#_{\theta}\, Q$ denotes the Riemannian metric%
\begin{equation*}
  (P \#_{\theta}\, Q)(x) := P(x) \#_{\theta}\, Q(x) \mbox{\quad for all\ } x \in K.%
\end{equation*}
The same inequality then also holds for the functions%
\begin{equation*}
  \Sigma_k:C^0(K,\SC^+_n) \rightarrow \R,\quad \Sigma_k(P) := \max_{x \in K}\Sigma_{k,x}(P).%
\end{equation*}
We are interested in minimizing one of the functions $\Sigma_k$ or their maximum $\max_{0 \leq k \leq n}\Sigma_k$, where $\Sigma_0(P) :\equiv 0$. To attack this problem via convex optimization, we have to restrict the domain to a finite-dimensional geodesically convex space. This can be achieved by considering only metrics of the form $P(x) = \rme^{r(x)}p$ with a polynomial function $r(x)$ satisfying $\deg r(x) \leq d$ for a given $d \in \N$ and $p\in \SC_n^+$. The space of all such metrics will be denoted by $C_d(K)$ and can be identified with $\R^N \tm \SC^+_n$, where $N = \binom{d+n}{n}$ is the number of coefficients in the polynomial. We note that $\R^N \tm \SC^+_n$ is a Hadamard manifold whose sectional curvature is bounded below by $-\frac{1}{2}$, see Proposition \ref{prop_sec_curv_lb}. This allows us to use different variants of the subgradient algorithm on manifolds for convex optimization problems on $\R^N \tm \SC^+_n$.%

We consider two problems in this paper: dimension and entropy estimation. In the first case, it is well-known (see, e.g., \cite[Thm.~9.1.1]{boichenko2005dimension}) 
that the Lyapunov, and thus the Hausdorff dimension of $K$ can be bounded by $\dim_L(K) \leq k + s$ for an integer $k \in \{0,1,\ldots,n-1\}$ and $s \in [0,1)$ provided that $K$ is backward-invariant and for some Riemannian metric $P$ the following inequality holds:%
\begin{equation*}
  \alpha_1^P(x)\alpha_2^P(x) \ldots \alpha_k^P(x) \alpha_{k+1}^P(x)^s < 1 \mbox{\quad for all\ } x \in K.%
\end{equation*}
Note that it is sufficient to verify this condition for all $x\in \tilde{K}$ in a larger set $\tilde{K}\supset K$, that is compact and forward invariant and thus contains an invariant set $K=\omega(\tilde{K})$.
To verify such a condition with our methods, we introduce the functions%
\begin{equation*}
  \Sigma_{k+s,x}(P) := \sum_{i=1}^k\log_2\alpha_i^P(x) + s\log_2\alpha_{k+1}^P(x)%
\end{equation*}
for any $k,s$ as above and $x \in K$. We can write $\Sigma_{k+s,x}(P) = s\Sigma_{k+1,x}(P) + (1-s)\Sigma_{k,x}(P)$, implying that $\Sigma_{k+s,x}$ satisfies the same convexity properties as described above. Hence, given any guess on $k$ and $s$, we can use the subgradient algorithm to minimize $P \mapsto \max_{x\in K}\Sigma_{k+s,x}(P)$ over $C_d(K)$ for some $d\in\N$ and check if the minimum is less than $0$. If this is the case, we can do the same for smaller choices of $k$ and $s$ in order to optimize the dimension estimate. If the minimum happens to be larger than $1$, we have to increase $d$ or $s$.%

The second problem is the upper estimation of restoration entropy, which measures the smallest data rate above which the system is regularly or finely observable on $K$ over a digital channel operating at this rate, see \cite{matveev2019observation} for precise definitions. By \cite{kawan2021remote}, the restoration entropy on $K$ satisfies%
\begin{equation*}
  h_{\res}(\phi,K) = \inf_{P \in C^0(K,\SC^+_n)}\max_{0 \leq k \leq n}\Sigma_k(P).%
\end{equation*}
Hence, we can again restrict the class of Riemannian metrics to $C_d(K)$ for some $d$ and minimize the convex function $\max\{\Sigma_k(P) : k =0,1,\ldots,n \}$ over $C_d(K)$.%

We fix a parametrization of the space of polynomials in $n$ variables and degree $\leq d$ over $\R^N$ with $N = \binom{d+n}{n}$, such that $r_a(x)$ denotes the polynomial with coefficient vector $a \in \R^N$. We then introduce the functions%
\begin{equation*}
  J_{k+s,x}(a,p) := \Sigma_{k+s,x}(\rme^{r_a(\cdot)}p),\quad J_{k+s,x}:\R^N \tm \SC^+_n \rightarrow \R%
\end{equation*}
for $k \in \{0,1,\ldots,n-1\}$ and $s \in [0,1)$. We further introduce%
\begin{equation*}
  J_{k+s}(a,p) := \max_{x\in K}J_{k+s,x}(a,p),\quad J_{k+s}:\R^N \tm \SC^+_n \rightarrow \R.%
\end{equation*}
Via the following proposition, we can reduce the computation of subgradients for our objective functions to the computation of subgradients of the functions $J_{k,x}$ with integer $k$. Its simple proof is omitted.%

\begin{proposition}
The following statements hold:%
\begin{enumerate}
\item[(i)] Assume that the maximum of $J_{k+s,x}(a,p)$ over $x \in K$ is attained at $x_*$. Then any subgradient $v$ of $J_{k+s,x_*}$ at $(a,p)$ is also a subgradient of $J_{k+s}$ at $(a,p)$.%
\item[(ii)] If $v,w \in \R^N \tm \SC_n$ are subgradients of $J_{k,x}$ and $J_{k+1,x}$ at $(a,p)$, respectively, then $(1 - s)v + sw$ is a subgradient of $J_{k+s,x}$ at $(a,p)$.%
\item[(iii)] If the maximum of $J_{k,x}(a,p)$, $k \in \{0,1,\ldots,n\}$, is attained at $k_*$, then any subgradient $v$ of $J_{k_*,x}$ at $(a,p)$ is a subgradient of $(a,p) \mapsto \max\{ J_{k,x}(a,p) : 0 \leq k \leq n \}$ at $(a,p)$.
\end{enumerate}
\end{proposition}

A straightforward computation shows that%
\begin{equation}\label{eq_j1a2}
  J_{k+s,x}(a,p) = \underbrace{\frac{k+s}{2\ln 2}(r_a(\phi(x)) - r_a(x))}_{=: J^1_{k+s,x}(a)} + \underbrace{\Sigma_{k+s,x}(p)}_{=: J^2_{k+s,x}(p)}.%
\end{equation}
It thus suffices to compute subgradients of $a \mapsto J^1_{k,x}(a)$ and $p \mapsto J^2_{k,x}(p)$ separately. Since $J^1_{k+s,x}$ is a linear function, the only subgradient is its constant gradient which is a polynomial in the components of $\phi(x)$ and $x$. The computation of a subgradient for $J^2_{k,x}$ is explained in \cite{kawan2021subgradient}. We briefly summarize the procedure how to compute it. Here, we use the notation $\zeta_x(p) := p^{\frac{1}{2}}A(x)p^{-\frac{1}{2}}$:%
\begin{enumerate}
\item[(1)] Compute an orthonormal basis $\{e_i\}$ of the tangent space $(\SC_n,\langle\cdot,\cdot\rangle_p)$.%
\item[(2)] Determine the solutions $X_i \in \SC_n$ of the Lyapunov equations%
\begin{equation*}
  p^{\frac{1}{2}}X_i + X_ip^{-\frac{1}{2}} = e_i.%
\end{equation*}
\item[(3)] Compute the matrices%
\begin{equation*}
  Z_i := \rmD\zeta_x(p)e_i = X_iA(x)p^{-\frac{1}{2}} - p^{\frac{1}{2}}A(x)p^{-\frac{1}{2}}X_ip^{-\frac{1}{2}}.%
\end{equation*}
\item[(4)] Compute a singular value decomposition%
\begin{equation*}
  \zeta_x(p) = U\trn \diag(\alpha_1(\zeta_x(p)),\ldots,\alpha_n(\zeta_x(p))) V.%
\end{equation*}
\item[(5)] Compute%
\begin{equation*}
  S_x := \frac{1}{\ln 2} U\trn \diag(\alpha_1(\zeta_x(p))^{-1},\ldots,\alpha_k(\zeta_x(p))^{-1},0,\ldots,0)V.%
\end{equation*}
\item[(6)] Compute a subgradient $v$ of $J^2_{k,x}$ by $v = \sum_i \tr[S_x\trn Z_i]e_i$.%
\end{enumerate}
To be precise, we can only guarantee that the above procedure works if there is a spectral gap, i.e.~$\alpha_k(\zeta_x(p)) > \alpha_{k+1}(\zeta_x(p))$.%

\section{Projected subgradient method on Hadamard manifolds}
\label{sec:3}

In this section, we introduce the projected subgradient method on Hadamard manifolds with inexact subgradients and analyze its convergence properties for different choices of step sizes. Later, we apply the method to the problem of singular value optimization for dynamical systems, where the Hadamard manifold is of the form $\MC = \R^N \tm \SC^+_n$ for some $n,N \in \N$ and the objective function is one of the functions $J_{k+s}$ or $\max\{ J_k : 0 \leq k \leq n \}$, as introduced in the preceding section.

Let $f\colon\MC \to \overline{\R}$ be a proper convex function. We consider the subgradient method to solve the optimization problem defined by%
\begin{equation}\label{eq:OptP}
	\min \{ f(p) \colon  p\in \mathcal{X} \},%
\end{equation}
where $\MC$ is a Hadamard manifold with lower bounded sectional curvature and $\XC \subseteq \MC$ is a compact convex set. {\it Throughout this section, we denote by $D$ an upper bound for the diameter of $\XC$, i.e., $\diam(\XC) \leq D$.}

The analysis presented here is based on \cite{FerreiraOliveira1998,FerreiraLouzeiroPrudente2019:subgrad}.%

\subsection{Preliminary results}\label{subsec:aux}

In this section, all functions $f:\MC \to\overline{\R}$ are assumed to be convex and lower semicontinuous on $\MC$.%

The following result is proved in \cite[Prop.~2.5]{WangLiWangYao2015}.%

\begin{proposition}\label{prop:subglim}
Let $\{p_k\} \subset \MC$ be a bounded sequence. If $\{s_k\}$ is a sequence satisfying $s_k \in \partial f(p_k)$ for each $k\in\N$, then $\{s_k\}$ is bounded as well.
\end{proposition}

Throughout this section, we let $\kappa$ be a lower bound on the sectional curvature of $\MC$ and define%
\begin{equation*}
	\hat{\kappa} := \sqrt{|\kappa|}.%
\end{equation*}
The following lemma plays an important role in the next subsections. Its proof will be omitted here, but it can be obtained, with some minor technical adjustments, by using Toponogov's theorem~\cite[p.~161, Thm.~4.2]{Sakai1996} and following the ideas of~\cite[Lem.~3.2]{WangLiWangYao2015}.%

\begin{lemma}\label{lem:comp}
Let $\MC$ be a Hadamard manifold with sectional curvature $\geq \kappa$. Let $p,q\in \MC$, $p\neq q$, $v\in T_p\MC$, $v\neq 0$, and ${\gamma}:[0,\infty)\to \MC$ be defined by ${\gamma}(t):=\exp_p\left(-tv\right/\|v\|)$. Then, for any $t\in[0,\infty)$, it holds that%
\begin{equation*}
	d^2({\gamma}(t),q)\leq d^2(p,q) +
	\frac{\sinh\left(t \hat{\kappa} \right)}{\hat{\kappa}} \left(\frac{\hat{\kappa}d(p,q)}{\tanh\left(\hat{\kappa}d(p,q)\right)}t+
	2\frac{\left\langle v,\log_pq  \right\rangle}{\left\|v\right\|}\right).%
\end{equation*}
\end{lemma}

\subsection{Inexact projected subgradient method}\label{sec:subgradient}

We denote by $f^*\coloneqq \inf_{ p \in \XC }f(p)$ the {\it optimal value} of problem \eqref{eq:OptP} and by $\Omega^*\coloneqq \{p\in \XC\colon f(p)=f^*\}$ its {\it solution set}. Since $f$ is lower semicontinuous and $\XC$ is compact, it is not difficult to see that $\Omega^* \neq \emptyset$.%

\subsubsection{Algorithm}

We propose to solve problem \eqref{eq:OptP} with the \emph{inexact Riemannian subgradient algorithm}, stated below.%


\begin{algorithm}[H]
	\begin{description}
		\item[ Step 0.] Let $p_0 \in \XC \subset \inner(\dom f)$. Set $k=0$.
		\item[ Step 1.] Choose a step size $t_k>0$ and compute%
		\begin{equation} \label{eq:SubGradMethod}
			 p_{k+1} = \PC_{\XC}(\tilde{p}_{k+1}), \qquad \tilde{p}_{k+1} := \exp_{p_k}\left(-t_k \frac{v_k}{\|v_k\|}\right),%
		\end{equation}
	   where $v_k \coloneqq s_k + u_k$ ($v_k\neq 0$) is an approximation for some $s_k\in \partial f(p_k)$.
		\item[ Step 2.] Set $k \leftarrow k+1$ and proceed to Step~1.%
	\end{description}
	\caption{}
	\label{alg:subgradient}
\end{algorithm}

\begin{lemma}\label{lem:MainIneq.exo}
Let $\sigma>0$ and $\zeta \coloneqq D\sinh\left(\hat{\kappa}\sigma\right)/(\sigma\tanh\left(\hat{\kappa}D\right))$. If $\sup_kt_k\leq \sigma$, then%
\begin{equation}\label{eq:lem.eq2.lemli}
  d^2(p_{k+1},q)\leq d^2(p_k,q) + \zeta t_k^2 + 2\frac{\sinh\left(\hat{\kappa}t_k\right)}{\hat{\kappa}} \cdot \frac{f(q) - f(p_k) + \|u_k\| D}{\left\|u_k + s_k\right\|}%
\end{equation}
for any $q \in \XC$ and $k\in\N_0$.
\end{lemma}

\begin{proof}
Take $q \in \XC$ and $k \in \N_0$. Applying Lemma~\ref{lem:comp} with $p=p_k$, $v=s_k + u_k$ and $t=t_k$, we obtain%
\begin{align*}
		&d^2(\tilde{p}_{k+1},q) \leq d^2(p_k,q) \\
		&+
		\frac{\sinh\left(\hat{\kappa}t_k\right)}{\hat{\kappa}}\left(\frac{\hat{\kappa} d(p_k,q)}{\tanh\left(\hat{\kappa} d(p_k,q)\right)} t_k +
		2\frac{\left\langle s_k,\log_{p_k}q \right\rangle + \left\langle u_k,\log_{p_k}q \right\rangle}{\left\|s_k + u_k\right\|}\right),%
\end{align*}
where $\tilde{p}_{k+1}$ is defined in \eqref{eq:SubGradMethod}. Using \eqref{eq:nonexp-proj}, \eqref{eq:defSugrad} and the Cauchy-Schwarz inequality, it follows that%
\begin{multline*}
	d^2(p_{k+1},q)\leq d^2(p_k,q) +
		\frac{\sinh\left(\hat{\kappa}t_k\right)}{\hat{\kappa}t_k} \cdot \frac{\hat{\kappa}d(p_k,q)}{\tanh\left(\hat{\kappa} d(p_k,q)\right)}t_k^2 \\
		+ 2\frac{\sinh\left(\hat{\kappa}t_k\right)}{\hat{\kappa}} \cdot \frac{f(q) - f(p_k) + \|u_k\|d(p_k,q)}{\left\|s_k +u_k\right\|} .
\end{multline*}
The maps $(0,\infty) \ni t \mapsto \sinh(t)/t$ and $(0,\infty) \ni t \mapsto t/\tanh(t)$ are positive and increasing. Thus, the last inequality above implies%
\begin{multline*}
	d^2(p_{k+1},q)\leq d^2(p_k,q) +
		\frac{\sinh\left(\hat{\kappa}\sigma \right)}{\hat{\kappa}\sigma} \cdot \frac{\hat{\kappa}D}{\tanh\left(\hat{\kappa}D\right)}t_k^2 \\
		+	2\frac{\sinh\left(\hat{\kappa}t_k\right)}{\hat{\kappa}} \cdot \frac{f(q) - f(p_k) + \|u_k\|D}{\left\|s_k + u_k\right\|}.%
\end{multline*}
It is not difficult to see that this inequality concurs with \eqref{eq:lem.eq2.lemli}. Therefore, the proof is complete.
\end{proof}

\subsection{Convergence analysis}

\subsubsection{Exogenous step size}

In this subsection, we assume that the sequence $\{p_k\}$ is generated by Algorithm~\ref{alg:subgradient} with exogenous step size, i.e., $t_k$ satisfies%
\begin{equation}\label{ExogenousStepsize}
	t_k > 0, \qquad \qquad \sum_{k=0}^{\infty}t_k = \infty, \qquad \qquad  \sum_{k=0}^{\infty}t_k^2 < \infty.%
\end{equation}

\begin{definition}\label{def:deltasolution}
For any $\delta > 0$, a point $q \in \XC$ is said to be a $\delta$-solution of \eqref{eq:OptP} if $f(q) \leq f^* + \delta$. The set of all $\delta$-solutions of \eqref{eq:OptP} is denoted by $\Omega^*_{\delta}$.
\end{definition}

\begin{theorem}\label{teo.Main}
Take $\epsilon > 0$. Suppose there is $\bar{k} \in \N$ such that $\|u_k\| \leq \epsilon$ for all $k \geq \bar{k}$. Then%
\begin{equation}\label{eq:linfs}
  \liminf_{k \rightarrow \infty} f(p_k) \leq f^* + \epsilon D.%
\end{equation}
In particular, $\{p_k\}$ has a cluster point that belongs to $\Omega^*_{\epsilon D}$.
\end{theorem}

\begin{proof}
Assume towards a contradiction that $\liminf_k f(p_k) > f^* + \epsilon D$. Then, there is $\beta > 1$ such that $\liminf_k f(p_k) > f^* + \epsilon D\beta$. Hence, using the definition of $f^*$ and $\liminf$, we know that there are $q' \in \XC$ and $k' \geq \bar{k}$ with%
\begin{equation*}
  f(q') < - \epsilon D + [\epsilon D (1-\beta) + \liminf_k f(p_k)] \leq  -\epsilon D + \inf_{k\geq k'}f(p_k).%
\end{equation*}
For better readability of the rest of the proof, we define%
\begin{equation}\label{teo:adef}
  a \coloneqq - f(q') - \epsilon D + \inf_{k\geq k'}f(p_k) > 0.%
\end{equation}
Since $\{p_k\}$ is bounded, by Proposition~\ref{prop:subglim} the sequence $\{s_k\}$ is also bounded. Considering a constant $\iota > 0$ such that $\left\|s_k\right\| \leq \iota$ for all $k\geq k'$, we have%
\begin{equation}\label{eq:teo.boundv}
  \|s_k + u_k\| \leq \|s_k\| + \|u_k\| \leq \iota + \epsilon, \qquad \forall\, k \geq k'.%
\end{equation}
On the other hand, Lemma~\ref{lem:MainIneq.exo} ensures that%
\begin{equation*}
	d^2(p_{k+1},q') \leq d^2(p_k,q') + \zeta t_k^2
	+	2t_k\frac{\sinh\left(\hat{\kappa}t_k\right)}{\hat{\kappa}t_k} \cdot \frac{f(q') - f(p_k) + \epsilon D}{\left\|u_k + s_k\right\|}, \quad k \in \N_0.%
\end{equation*}
Using \eqref{teo:adef}, \eqref{eq:teo.boundv} and the fact that the function $(0,\infty) \ni t \mapsto \sinh(t)/t$ is bounded below by $1$, the last inequality implies%
\begin{equation*}
	d^2(p_{k+1},q') \leq d^2(p_k, q') + \zeta t_k^2 -	t_k\frac{2a}{\iota + \epsilon}, \qquad k\geq k'.%
\end{equation*}
Consider $\ell \in \N$. Thus, from the last inequality, summing over $k$, we obtain%
\begin{equation*}
  \frac{2a}{\iota + \epsilon}\sum_{k=k'}^{k' + \ell}t_k \leq d^2(p_{k'},q') - d^2(p_{k' + \ell + 1},q') + \zeta \sum_{k=k'}^{k' + \ell}t^2_k \leq d^2(p_{k'},q') + \zeta\sum_{k=k'}^{k' + \ell}t^2_k.%
\end{equation*}
Since the last inequality holds for all $\ell \in \N$, by using the inequality in \eqref{ExogenousStepsize}, we arrive at a contradiction. Therefore, \eqref{eq:linfs} holds.%

Note that inequality \eqref{eq:linfs} implies that $\{f(p_k)\}$ has a monotonically decreasing subsequence $\{f(p_{k_j})\}$ such that $\lim_{j\rightarrow \infty}f(p_{k_j}) \leq f^* + \epsilon D$. Since $\{p_{k_j}\}$ is bounded, there are $p_* \in \XC$ and a subsequence $\{p_{k_{\ell}}\} \subseteq \{p_{k_j}\}$ such that $\lim_{\ell\rightarrow\infty}p_{k_\ell} = p_*$, which by the lower semicontinuity of $f$ implies $f(p_*) \leq \lim_{\ell\rightarrow\infty}f(p_{k_\ell}) \leq f^* + \epsilon D$, and then $p_* \in \Omega^*_{\epsilon D}$.
\end{proof}

The next result allows us to obtain a number of iterations $N$ that guarantees the existence of a point $p_k$ sufficiently close to an $\epsilon D$-solution of \eqref{eq:OptP}.%

\begin{theorem}\label{eq:icomp}
Take $\epsilon > 0$, $\iota > 0$, $D>0$ and $N \in \N$. Assume that $\diam(\XC) < D$. If $\|u_k\| \leq \epsilon$ and $\|s_k\| \leq \iota$ for all $k \in \{0,1,\ldots,N\}$, then%
\begin{equation}\label{eq:complexityExogenous}
  \min \left\{f(p_k) - (f^* + \epsilon D) \colon k=0, 1,\ldots, N\right\} \leq (\epsilon + \iota) \frac{D^2 + \zeta \sum_{k=0}^{N} t_k^2}{2\sum_{k=0}^{N} t_k}.%
\end{equation}
\end{theorem}

\begin{proof}
If $p_k \in \Omega^*_{\epsilon D}$ for some $k \in \{0,1,\ldots, N\}$, the definition of $\Omega^*_{\epsilon D}$ implies that $f(p_k) - (f^* + \epsilon D) \leq 0$, and \eqref{eq:complexityExogenous} is trivially satisfied. Therefore, we only need to consider the case that%
\begin{equation}\label{eq:pkoutsideomega}
	p_k \notin \Omega^*_{\epsilon D}, \qquad \forall k \in \{0,1,\ldots, N\}.%
\end{equation}
Pick $q \in \Omega^*$. Using that $\|u_k\| \leq \epsilon$ for all $k \in \{0,1,\ldots,N\}$ and applying Lemma~\ref{lem:MainIneq.exo}, we obtain%
\begin{equation*}
	d^2(p_{k+1}, q )\leq d^2(p_k, q ) + \zeta t_k^2
		+ 2 t_k	\frac{\sinh\left(\hat{\kappa}t_k\right)}{\hat{\kappa}t_k} \cdot \frac{(f^* + \epsilon D ) - f(p_k)}{\left\|u_k + s_k\right\|}%
\end{equation*}
for $k = 0,1,\ldots,N$. By \eqref{eq:pkoutsideomega}, $(f^* + \epsilon D) - f(p_k) < 0$ holds for any $k \in \{0,1,\ldots,N\}$. Hence, using that the function $(0,\infty) \ni t \mapsto \sinh(t)/t$ is bounded below by $1$, we obtain%
\begin{equation*}
  2t_k\frac{f(p_k) - (f^* + \epsilon D)}{\|u_k + s_k\|} \leq d^2(p_k, q ) - d^2(p_{k+1},q) + \zeta t_k^2 , \quad k=0,1,\ldots,N.%
\end{equation*}
Summing over $k=0,1,\ldots,N$, we obtain%
\begin{equation*}
  2\sum_{k=0}^N t_k \frac{f(p_k) - (f^*  +  \epsilon D)}{\|u_k + s_k\|} \leq d^2(p_0, q)- d^2(p_{N+1}, q) + \zeta \sum_{k=0}^N t_k^2.%
\end{equation*}
Since $\|u_k + s_k\| \leq \| u_k\| + \|s_k\| < \epsilon + \iota$ holds for any $k \in \{0,1,\ldots,N\}$, the last inequality implies%
\begin{equation*}
  \frac{2}{\epsilon + \iota} \min\left\{f(p_k) - (f^* + \epsilon D)\colon  k=0,1,\ldots,N\right\}\sum_{k=0}^N t_k \leq D^2 + \zeta\sum_{k=0}^N t_k^2,%
\end{equation*}
which is	equivalent to the desired inequality.
\end{proof}

\subsubsection{Constant step size}

Here, the sequence $\{p_k\}$ is generated by Algorithm~\ref{alg:subgradient} with constant step size, i.e., $t_k$ satisfies%
\begin{equation}\label{ConstantStepsize}
	\bar{t} = \mbox{positive constant}, \qquad t_k = \bar{t}, \quad \forall k.%
\end{equation}
Since Lemma~\ref{lem:MainIneq.exo} holds for $t_k$ given in \eqref{ConstantStepsize}, the proof of the next result is very similar to the proof of Theorem~\ref{eq:icomp}, so it will be omitted. In \cite[Thm.~9]{ZhangSra2016}, the authors propose a version of this theorem using the exact subgradient.%

\begin{theorem}\label{eq:icomp.const}
Take $\epsilon > 0$, $\iota > 0$ and $N \in \N$. If $\|u_k\| \leq \epsilon$ and $\|s_k\| \leq \iota$ for all $k \in \{0,1,\ldots,N-1\}$, then%
\begin{equation}\label{eq:complexityConstant}
	\min \left\{  f(p_k) - (f^* + \epsilon D)   \colon k=0, 1,\ldots, N-1\right\}
	\leq (\epsilon + \iota) \frac{D^2 + \zeta N \bar{t}^2}{2 N \bar{t}}.%
\end{equation}
In particular, if $\bar{t} = D/(\sqrt{\zeta N})$, we have%
\begin{equation}\label{eq:tol.constant.stepsize}
	\min \left\{f(p_k) - (f^* + \epsilon D) \colon k=0,1,\ldots, N-1\right\} \leq (\epsilon + \iota) \frac{D\sqrt{\zeta}}{\sqrt{N}}.%
\end{equation}
\end{theorem}

\subsubsection{Polyak step size}

Take $\epsilon > 0$. Define the inexact version of the Polyak step size by%
\begin{equation}\label{StepsizePolyak}
	t_k = \alpha\frac{f(p_k)-f^*-\epsilon D}{\left\|v_k\right\|},\qquad 0 < \alpha < 2\frac{\tanh\left(\hat{\kappa}D \right)}{\hat{\kappa }D}.%
\end{equation}

\begin{theorem}\label{teo.complexity.polyak}
Take $\iota > 0$ and $N \in \N$. If $\|u_k\| \leq \epsilon$, $\|s_k\| \leq \iota$ and $t_k \geq 0$ for all $k \in \{0,1,\ldots,N-1\}$, then%
\begin{equation} \label{eq:complexityPolyak}
  \min\left\{f(p_k) - (f^* + \epsilon D) \colon k=0,1,\ldots, N-1\right\} \leq \frac{1}{\sqrt{\Gamma}\sqrt{N}},%
\end{equation}
where $\Gamma$ is the constant defined by%
\begin{equation}\label{eq:Gamma.def}
  \Gamma = \frac{1}{(\epsilon + \iota)^2D^2}\left(2\alpha - \frac{\hat{\kappa}D}{\tanh\left(\hat{\kappa}D\right)}\alpha^2\right).%
\end{equation}
\end{theorem}

\begin{proof}
Pick $q \in \Omega^*$ and $k \in \{0,1,\ldots,N-1\}$. Applying Lemma~\ref{lem:comp} with $p=p_k$, $v=v_k \coloneqq s_k + u_k$ and $t=t_k$, we obtain%
\begin{align*}
	&d^2(\tilde{p}_{k+1},q) \leq d^2(p_k,q) \\
	&+\frac{\sinh\left(\hat{\kappa}t_k\right)}{ \hat{\kappa}t_k}
		\left(\frac{\hat{\kappa}d(p_k,q)}{\tanh\left(\hat{\kappa} d(p_k,q)\right)} t^2_k +
		2t_k\frac{\left\langle  s_k ,\log_{p_k}q  \right\rangle + \left\langle u_k,\log_{p_k}q \right\rangle}{\left\|v_k\right\|}\right),%
\end{align*}
where $\tilde{p}_{k+1}$ is defined in \eqref{eq:SubGradMethod}. Hence, using \eqref{eq:nonexp-proj}, \eqref{eq:defSugrad} and the Cauchy-Schwarz inequality, it follows that%
\begin{align*}
	&d^2(p_{k+1},q) \leq d^2(p_k,q) \\
	&+	\frac{\sinh\left(\hat{\kappa}t_k\right)}{\hat{\kappa}t_k}\left(\frac{\hat{\kappa} d(p_k,q)}{\tanh\left(\hat{\kappa} d(p_k,q)\right)} t_k^2 - 2t_k\frac{f(p_k) - f^* - \|u_k\|d(p_k,q)}{\left\|v_k\right\|}\right).%
\end{align*}
The function $(0,\infty) \ni t \mapsto t/\tanh(t)$ is positive and increasing, and $d(p_k,q)\leq D$ holds. Thus, the last inequality above implies%
\begin{equation*}
	d^2(p_{k+1},q)\leq d^2(p_k,q) +
		\frac{\sinh\left(\hat{\kappa}t_k\right)}{\hat{\kappa}t_k}\left(\frac{\hat{\kappa}D}{\tanh\left(\hat{\kappa}D\right)}t_k^2 - 2t_k\frac{   f(p_k) - f^* - \epsilon D}{\left\|v_k\right\|}\right).%
\end{equation*}
Since $\sinh(\hat{\kappa}t_k)/(\hat{\kappa}t_k) \geq 1$ and $\|v_k\| \leq \epsilon + \iota $, after rearranging the terms of the last inequality and using \eqref{StepsizePolyak} and \eqref{eq:Gamma.def}, we obtain%
\begin{align*}
	\Gamma D^2 \left(f(p_k) - f^* - \epsilon D\right)^2 &\leq \Gamma D^2( \epsilon + \iota )^2 \left(
		\frac{f(p_k) - f^* - \epsilon D}{\left\|v_k\right\|}\right)^2 \\
		&\leq d^2(p_k,q) - d^2(p_{k+1},q).%
\end{align*}
Summing over $k=0,1,\ldots,N-1$, we arrive at%
\begin{equation*}
	\Gamma D^2\sum_{k=0}^{N-1}\left(f(p_k) - f^* - \epsilon D\right)^2 \leq d^2(p_0,q) - d^2(p_N,q) \leq D^2.%
\end{equation*}
Therefore,%
\begin{equation*}
  \min \left\{(f(p_k) - f^* - \epsilon D)^2 \colon k=0,1,\ldots,N-1\right\} \leq \frac{1}{\Gamma N}.%
\end{equation*}
which is equivalent to the desired inequality.
\end{proof}

\section{Some auxiliary results}\label{sec:4}

In this section, we provide some auxiliary results needed for the application of the projected inexact subgradient method to the singular value optimization problems described in Subsection \ref{subsec_problem}.%

\subsection{Projections to order intervals}

For any $p_1,p_2 \in \SC^+_n$ with $p_1 \leq p_2$, we define the order interval%
\begin{equation*}
  [p_1,p_2] := \{ p \in \SC^+_n : p_1 \leq p \leq p_2 \}.%
\end{equation*}

\begin{lemma}
Every order interval $[p_1,p_2]$ is a compact and convex subset of $\SC^+_n$.%
\end{lemma}


\begin{proof}
If $p,q \in [p_1,p_2]$, then by \cite[Prop.~4.1]{lim2014weighted} we have $p_1 = p_1 \#\, p_1 \leq p \#\, q \leq p_2 \#\, p_2 = p_2$, which shows that $[p_1,p_2]$ is convex. It is easy to see that $[p_1,p_2]$ is closed. To prove compactness, it hence suffices to show boundedness (since $\SC^+_n$ is a complete Riemannian manifold). To this end, take an arbitrary $p \in [p_1,p_2]$. By \cite[Cor.~I.2.1.1]{boichenko2005dimension}, the relation $p_1 \leq p \leq p_2$ implies $\lambda_i(p) \in [\lambda_i(p_1),\lambda_i(p_2)]$ for $i = 1,\ldots,n$. It follows that%
\begin{align*}
  d(p,I)^2 &= \sum_{i=1}^n \ln^2 \lambda_i(p) = \sum_{\lambda_i(p) < 1} \ln^2 \lambda_i(p) + \sum_{\lambda_i(p) > 1} \ln^2 \lambda_i(p) \\
	&\leq \sum_{\lambda_i(p) < 1} \ln^2 \lambda_i(p_1) + \sum_{\lambda_i(p) > 1} \ln^2 \lambda_i(p_2) \leq d(p_1,I)^2 + d(p_2,I)^2.%
\end{align*}
Here, we use that $\ln^2$ is decreasing on $(0,1]$ and increasing on $[1,\infty)$.%
\end{proof}

From now on, we only consider order intervals of the form $[\alpha I,\beta I]$ for $0 < \alpha < \beta$.%

\begin{lemma}
The projection onto the set $[\alpha I,\beta I]$, denoted by $\PC_{\alpha,\beta}$, is given by%
\begin{equation*}
  \PC_{\alpha,\beta}(p) = U_p\trn \diag(\zeta_1(p),\ldots,\zeta_n(p)) U_p,%
\end{equation*}
where $U_p$ is an orthogonal matrix chosen such that $U_ppU_p\trn$ is diagonal and%
\begin{equation*}
  \zeta_i(p) := \left\{ \begin{array}{rl}
	                                                                          \alpha & \mbox{if } \lambda_i(p) < \alpha, \\
																																						\lambda_i(p) & \mbox{if } \lambda_i(p) \in [\alpha,\beta], \\
																																						\beta & \mbox{if } \lambda_i(p) > \beta.%
																																						\end{array}\right.%
\end{equation*}
\end{lemma}

\begin{proof}
Let $p \in \SC^+_n$ be arbitrary. To compute $\PC_{\alpha,\beta}(p) \in [\alpha I,\beta I]$, we have to solve the minimization problem%
\begin{equation*}
  \min_{\alpha I \leq q \leq \beta I} d(p,q)^2.%
\end{equation*}
This problem can equivalently be written as%
\begin{equation*}
  \min_{\lambda_i(q) \in [\alpha,\beta]} \sum_{i=1}^n \ln^2 \lambda_i(p^{-1}q).%
\end{equation*}
We write $p = U\trn \diag(\gamma_1,\ldots,\gamma_n) U$ for some orthogonal $U$ with $\gamma_1 \geq \cdots \geq \gamma_n$, and claim that the minimizer $q^*$ is given by%
\begin{equation*}
  q^* := U\trn \diag(\zeta_1,\ldots,\zeta_n) U,\quad \zeta_i := \left\{ \begin{array}{cl}
	                                                                          \alpha & \mbox{if } \gamma_i < \alpha, \\
																																						\gamma_i & \mbox{if } \gamma_i \in [\alpha,\beta], \\
																																						\beta & \mbox{if } \gamma_i > \beta.%
																																						\end{array}\right.%
\end{equation*}
By construction, $q^*$ is an element of $[\alpha I,\beta I]$. Its distance to $p$ satisfies%
\begin{align*}
  d(p,q^*)^2 &= \sum_{i=1}^n \ln^2 \lambda_i(p^{-1}q^*) = \sum_{\gamma_i < \alpha} \ln^2 \frac{\alpha}{\gamma_i} + \sum_{\gamma_i > \beta} \ln^2 \frac{\beta}{\gamma_i}.%
\end{align*}
To show that $q^*$ is the (unique) minimizer, it suffices to prove that $d(p,q) \geq d(p,q^*)$ for every $q \in [\alpha I,\beta I]$. To this end, pick $q \in [\alpha I,\beta I]$ arbitrary and define $\tilde{q}_1 := q - \alpha I \geq 0$, $\tilde{q}_2 := \beta I - q \geq 0$. Then%
\begin{align*}
  d(p,q)^2 &= \sum_{i=1}^n \ln^2 \lambda_i(p^{-1}q)  \\
	&\geq \sum_{\lambda_i(p^{-1}) > \alpha^{-1}} \ln^2 \lambda_i(p^{-1}(\tilde{q}_1 + \alpha I)) + \sum_{\lambda_i(p^{-1}) < \beta^{-1}} \ln^2 \lambda_i(p^{-1}(\beta I - \tilde{q}_2)) \\
	&= \sum_{\lambda_i(p^{-1}) > \alpha^{-1}} \ln^2 \lambda_i(\alpha p^{-1} + p^{-\frac{1}{2}}\tilde{q}_1p^{-\frac{1}{2}}) \\
	&\qquad + \sum_{\lambda_i(p^{-1}) < \beta^{-1}} \ln^2 \lambda_i(\beta p^{-1} - p^{-\frac{1}{2}}\tilde{q}_2p^{-\frac{1}{2}}) \\
	&\geq \sum_{\lambda_i(p^{-1}) > \alpha^{-1}} \ln^2 \left[\alpha \lambda_i(p^{-1})\right] + \sum_{\lambda_i(p^{-1}) < \beta^{-1}} \ln^2 \left[\beta \lambda_i(p^{-1})\right] \\
	&= \sum_{\lambda_{n-i}(p) < \alpha} \ln^2\frac{\alpha}{\lambda_{n-i}(p)} + \sum_{\lambda_{n-i}(p) > \beta} \ln^2 \frac{\beta}{\lambda_{n-i}(p)} = d(p,q^*)^2.%
\end{align*}
In the last inequality, we used that $p^{-\frac{1}{2}}\tilde{q}_1p^{-\frac{1}{2}} \geq 0$, $p^{-\frac{1}{2}}\tilde{q}_2p^{-\frac{1}{2}} \geq 0$ in combination with \cite[Cor.~I.2.1.1]{boichenko2005dimension} and the monotonicity properties of $\ln^2$.%
\end{proof}

\begin{lemma}\label{lem_diam}
The diameter of the set $[\alpha I,\beta I]$ is given by%
\begin{equation*}
  \diam\, [\alpha I,\beta I] = \sqrt{n} \ln \frac{\beta}{\alpha}.%
\end{equation*}
\end{lemma}

\begin{proof}
Let $p,q \in [\alpha I,\beta I]$ be chosen arbitrarily. Then%
\begin{align*}
  d(p,q)^2 &= \sum_{i=1}^n \ln^2 \lambda_i(p^{-1}q) \\
	         &= \sum_{\lambda_i(p^{-1}q) < 1} \ln^2 \lambda_i(p^{-1}q) + \sum_{\lambda_i(p^{-1}q) > 1} \ln^2 \lambda_i(p^{-1}q) \\
					&\leq \sum_{\lambda_i(p^{-1}q) < 1} \ln^2 \lambda_n(p^{-1}q) + \sum_{\lambda_i(p^{-1}q) > 1} \ln^2 \lambda_1(p^{-1}q) \\
					&= \sum_{\lambda_i(p^{-1}q) < 1} \ln^2 \lambda_1(q^{-1}p) + \sum_{\lambda_i(p^{-1}q) > 1} \ln^2 \lambda_1(p^{-1}q).%
\end{align*}
Using that $\lambda_1(AB) \leq \lambda_1(A)\lambda_1(B)$ for positive definite $A,B$ (\cite[Lem.~4.2]{weber2017riemannian}), we get%
\begin{align*}
  d(p,q)^2 &\leq \sum_{\lambda_i(p^{-1}q) < 1} \ln^2\{\lambda_1(q^{-1})\lambda_1(p)\} + \sum_{\lambda_i(p^{-1}q) > 1} \ln^2 \{\lambda_1(p^{-1})\lambda_1(q)\} \\
	&\leq \sum_{\lambda_i(p^{-1}q) < 1} \ln^2 \frac{\beta}{\alpha} + \sum_{\lambda_i(p^{-1}q) > 1} \ln^2 \frac{\beta}{\alpha} \leq n \ln^2 \frac{\beta}{\alpha}.%
\end{align*}
Since $d(\alpha I,\beta I) = \sqrt{n}\ln\frac{\beta}{\alpha}$, the proof is complete.%
\end{proof}

\subsection{Lipschitz constant}

In this subsection, we derive a Lipschitz constant of the functions $J^2_{k+s,x}$.%

\begin{lemma}\label{lem_lip}
For each $x \in K$, $k \in \{1,\ldots,n-1\}$ and $s \in [0,1)$, the function $J^2_{k+s,x}:\SC^+_n \rightarrow \R$, introduced in \eqref{eq_j1a2}, is globally Lipschitz continuous with Lipschitz constant $L = \sqrt{n}/\ln(2)$.%
\end{lemma}

\begin{proof}
Since $J^2_{k+s,x} = (1-s)J^2_{k,x} + sJ^2_{k+1,x}$, it suffices to prove the statement for $s = 0$. Using the identities%
\begin{align*}
  \lambda_i(p^{-1}q) = \lambda_i( p^{-\frac{1}{2}} q p^{-\frac{1}{2}} ) = \lambda_i((p^{-\frac{1}{2}} q^{\frac{1}{2}})(p^{-\frac{1}{2}}q^{\frac{1}{2}})\trn) = \alpha_i(p^{-\frac{1}{2}}q^{\frac{1}{2}})^2%
\end{align*}
and formula \eqref{eq_snp_dist}, we can rewrite $d(p,q)$ as%
\begin{equation*}
  d(p,q) = 2\Bigl( \sum_{i=1}^n [\ln \alpha_i(p^{-\frac{1}{2}}q^{\frac{1}{2}})]^2 \Bigr)^{\frac{1}{2}}.%
\end{equation*}
Let $\omega_0(g) := 1$ and $\omega_k(g) := \alpha_1(g)\cdots \alpha_k(g)$ for any $g\in\GL(n,\R)$ and $k = 1,\ldots,n$. Then%
\begin{align*}
  J^2_{k,x}(p) - J^2_{k,x}(q) &= \sum_{i=1}^k \log_2 \alpha_i(p^{\frac{1}{2}}A(x)p^{-\frac{1}{2}}) - \sum_{i=1}^k \log_2 \alpha_i(q^{\frac{1}{2}}A(x)q^{-\frac{1}{2}}) \\
	&= \log_2 \omega_k(p^{\frac{1}{2}}A(x)p^{-\frac{1}{2}}) - \log_2 \omega_k(q^{\frac{1}{2}}A(x)q^{-\frac{1}{2}}) \\
	&= \log_2 \frac{\omega_k(p^{\frac{1}{2}}A(x)p^{-\frac{1}{2}})}{\omega_k(q^{\frac{1}{2}}A(x)q^{-\frac{1}{2}})} = \log_2 \frac{\omega_k(p^{\frac{1}{2}} q^{-\frac{1}{2}} q^{\frac{1}{2}}A(x)q^{-\frac{1}{2}} q^{\frac{1}{2}}p^{-\frac{1}{2}})}{\omega_k(q^{\frac{1}{2}}A(x)q^{-\frac{1}{2}})}.%
\end{align*}
Using Horn's inequality $\omega_k(AB) \leq \omega_k(A)\omega_k(B)$ (see \cite[Prop.~I.2.3.1]{boichenko2005dimension}), we thus obtain%
\begin{align*}
  J^2_{k,x}(p) - J^2_{k,x}(q) \leq \frac{1}{\ln(2)} \ln [\omega_k(p^{\frac{1}{2}}q^{-\frac{1}{2}}) \omega_k(q^{\frac{1}{2}}p^{-\frac{1}{2}})].%
\end{align*}
Writing $\|\cdot\|_1$ and $\|\cdot\|_2$ for the $1$-norm and $2$-norm in $\R^n$, respectively, we can further estimate%
\begin{align*}
  \ln\omega_k(p^{\frac{1}{2}}q^{-\frac{1}{2}}) &= \sum_{i=1}^k \ln \alpha_i(p^{\frac{1}{2}}q^{-\frac{1}{2}}) \leq \sum_{i=1}^k |\ln \alpha_i(p^{\frac{1}{2}}q^{-\frac{1}{2}})| \leq \ln(2) \| \vec{\sigma}(p^{\frac{1}{2}}q^{-\frac{1}{2}}) \|_1 \\
	&\leq \sqrt{n} \ln(2) \| \vec{\sigma}(p^{\frac{1}{2}}q^{-\frac{1}{2}}) \|_2 = \frac{\sqrt{n}}{2} d(p^{-1},q^{-1}) = \frac{\sqrt{n}}{2} d(p,q).%
\end{align*}
The analogous estimates hold for $\omega_k(q^{\frac{1}{2}}p^{-\frac{1}{2}})$, and hence%
\begin{equation*}
  |J^2_{k,x}(p) - J^2_{k,x}(q)| \leq \frac{\sqrt{n}}{\ln(2)}\, d(p,q).%
\end{equation*}
The proof is complete.
\end{proof}

\subsection{Error analysis}

Assume that $s_2(x)$ is the subgradient of $J^2_{k,x}$ at $p \in \SC^+_n$ for some $k \in \{1,\ldots,n-1\}$ and $x \in K$ as computed via the procedure explained at the end of Subsection \ref{subsec_problem}. Our aim is to understand how far $s_2(x)$  can be from a real subgradient in the case that $x$ is not a maximizer of $x \mapsto J^2_{k,x}(p)$. Recall that $\zeta_x(p) = p^{\frac{1}{2}}A(x)p^{-\frac{1}{2}}$.%

The following theorem helps to compute the error in case that $x$ is close to a maximizer.%

\begin{theorem}
Let $x,y \in K$ and assume that%
\begin{equation}\label{eq_def_delta}
  \delta := \alpha_k(\zeta_y(p)) - \alpha_{k+1}(\zeta_x(p)) > 0.%
\end{equation}
Then%
\begin{align*}
  \|s_2(x) - s_2(y)\|_p \leq \frac{\sqrt{2n(n+1)}}{\ln 2} \frac{1}{\delta} \|p^{\frac{1}{2}}\|_F \|p^{-\frac{1}{2}}\|_F \|A(x) - A(y)\|_F.%
\end{align*}
\end{theorem}

\begin{proof}
By linearity of the trace function and orthonormality of $\{e_i\}$, we have%
\begin{align*}
  \|s_2(x) - s_2(y)\|^2_p &= \Bigl\| \sum_i \tr[S_x\trn \rmD\zeta_x(p)e_i - S_y\trn \rmD\zeta_y(p)e_i]e_i \Bigr\|_p^2 \\
	                        &= \sum_i \tr^2[S_x\trn \rmD\zeta_x(p)e_i - S_y\trn \rmD\zeta_y(p)e_i].%
\end{align*}
Using again the linearity of the trace function and the fact that $\tr[AB] = \tr[BA]$ for any $A,B$, we find that%
\begin{align*}
  \tr[S_x\trn \rmD\zeta_x(p)e_i] &= \tr[S_x\trn(X_iA(x)p^{-\frac{1}{2}} - p^{\frac{1}{2}}A(x)p^{-\frac{1}{2}}X_ip^{-\frac{1}{2}})] \\
	&= \tr[X_iA(x)p^{-\frac{1}{2}}S_x\trn - X_ip^{-\frac{1}{2}} S_x\trn p^{\frac{1}{2}}A(x)p^{-\frac{1}{2}}] \\
	&= \tr [X_ip^{-\frac{1}{2}} ( \zeta_x(p) S_x\trn - S_x\trn \zeta_x(p) ) ].%
\end{align*}
Now write $\alpha_i := \alpha_i(\zeta_x(p))$ and observe that%
\begin{align*}
  \zeta_x(p) S_x\trn &= \frac{1}{\ln 2} U\trn \diag(\alpha_1,\ldots,\alpha_n) V V\trn \diag(\alpha_1^{-1},\ldots,\alpha_k^{-1},0,\ldots,0) U \\
	&= \frac{1}{\ln 2} U\trn (I_{k\tm k} \oplus 0_{(n-k)\tm (n-k)}) U,%
\end{align*}
and analogously%
\begin{align*}
  S_x\trn \zeta_x(p) = \frac{1}{\ln 2} V\trn (I_{k\tm k} \oplus 0_{(n-k)\tm (n-k)}) V.%
\end{align*}
We observe that $U\trn (I_{k\tm k} \oplus 0_{(n-k)\tm (n-k)}) U$ is the orthogonal projection onto the $k$-dimensional subspace of $\R^n$ spanned by the first $k$ left singular vectors of $\zeta_x(p)$, that we denote by $u_1(x) := U_x\trn e_1,\ldots,u_k(x) := U_x\trn e_k$. Introducing the notation $\bar{U}_x := [u_1(x)|\cdots|u_k(x)] \in \R^{n \tm k}$, we further observe that%
\begin{equation*}
  U\trn(I_{k\tm k} \oplus 0_{(n-k)\tm (n-k)})U = \bar{U}_x\bar{U}_x\trn.%
\end{equation*}
Altogether, we have obtained the identity%
\begin{align*}
  \|s_2(x) - s_2(y)\|^2_p = \frac{1}{(\ln 2)^2}\sum_i \tr^2[X_ip^{-\frac{1}{2}}(\bar{U}_x\bar{U}_x\trn - \bar{V}_x\bar{V}_x\trn - \bar{U}_y\bar{U}_y\trn + \bar{V}_y\bar{V}_y\trn)].%
\end{align*}
To estimate the trace, we use the Cauchy-Schwarz inequality in $\R^{n\tm n}$ which yields%
\begin{equation*}
  \|s_2(x) - s_2(y)\|^2_p \leq \frac{1}{(\ln 2)^2} \sum_i \|X_ip^{-\frac{1}{2}}\|_F^2\|\bar{U}_x\bar{U}_x\trn - \bar{V}_x\bar{V}_x\trn - \bar{U}_y\bar{U}_y\trn + \bar{V}_y\bar{V}_y\trn\|_F^2.%
\end{equation*}
To estimate the term $\|X_ip^{-\frac{1}{2}}\|_F$, we use that%
\begin{align*}
  1 &= \|e_i\|_p^2 = \tr[p^{-1}e_ip^{-1}e_i] = \tr[p^{-1}(p^{\frac{1}{2}}X_i + X_ip^{\frac{1}{2}}) p^{-1}(p^{\frac{1}{2}}X_i + X_ip^{\frac{1}{2}})] \\
	&= \tr[(p^{-\frac{1}{2}}X_i + p^{-1}X_ip^{\frac{1}{2}})(p^{-\frac{1}{2}}X_i + p^{-1}X_ip^{\frac{1}{2}})] \\
	&= \tr[p^{-\frac{1}{2}}X_ip^{-\frac{1}{2}}X_i] + \tr[p^{-\frac{1}{2}}X_ip^{-1}X_ip^{\frac{1}{2}}] + \tr[p^{-1}X_i^2] \\
	&\quad + \tr[p^{-1}X_ip^{-\frac{1}{2}}X_ip^{\frac{1}{2}}] \\
	&= 2 \tr[p^{-\frac{1}{2}}X_ip^{-\frac{1}{2}}X_i] + 2 \tr[X_ip^{-1}X_i] \\
	&= 2 \|X_i\|^2_{p^{\frac{1}{2}}} + 2 \tr[ (X_ip^{-\frac{1}{2}})\trn(X_ip^{-\frac{1}{2}}) ]%
\end{align*}
which implies%
\begin{equation*}
  \|X_ip^{-\frac{1}{2}}\|^2_F = \tr[ (X_ip^{-\frac{1}{2}})\trn(X_ip^{-\frac{1}{2}}) ] = \frac{1}{2} - \|X_i\|_{p^{\frac{1}{2}}}^2 \leq \frac{1}{2}.%
\end{equation*}
Hence, we obtain%
\begin{align*}
  \|s_2(x) - s_2(y)\|^2_p &\leq \frac{n(n+1)}{4(\ln 2)^2} \|\bar{U}_x\bar{U}_x\trn - \bar{V}_x\bar{V}_x\trn - \bar{U}_y\bar{U}_y\trn + \bar{V}_y\bar{V}_y\trn\|_F^2 \\
	&\leq \frac{n(n+1)}{4(\ln 2)^2}\bigl( \|\bar{U}_x\bar{U}_x\trn - \bar{U}_y\bar{U}_y\trn\|_F + \|\bar{V}_x\bar{V}_x\trn - \bar{V}_y\bar{V}_y\trn\|_F \bigr)^2.%
\end{align*}
Next, we use that $\|\bar{U}_x\bar{U}_x\trn - \bar{U}_y\bar{U}_y\trn\|_F$ is the distance between the two subspaces $L^U_x := \langle u_1(x),\ldots,u_k(x) \rangle$ and $L^U_y := \langle u_1(y),\ldots,u_k(y) \rangle$, regarded as elements of the Grassmannian of $k$-planes in $\R^n$. We further obtain%
\begin{align*}
  \|s_2(x) - s_2(y)\|_p &\leq \frac{\sqrt{n(n+1)}}{2\ln 2}\bigl( \|\bar{U}_x\bar{U}_x\trn - \bar{U}_y\bar{U}_y\trn\|_F + \|\bar{V}_x\bar{V}_x\trn - \bar{V}_y\bar{V}_y\trn\|_F\bigr) \\
	&\leq \frac{\sqrt{2n(n+1)}}{2\ln 2} \sqrt{\|\bar{U}_x\bar{U}_x\trn - \bar{U}_y\bar{U}_y\trn\|_F^2 + \|\bar{V}_x\bar{V}_x\trn - \bar{V}_y\bar{V}_y\trn\|_F^2}.%
\end{align*}
We have%
\begin{equation*}
  \|\bar{U}_x\bar{U}_x\trn - \bar{U}_y\bar{U}_y\trn\|_F = \sqrt{2}\|\sin \Theta(L^U_x,L^U_y)\|_F,%
\end{equation*}
where $\Theta$ is the matrix of canonical angles between $L_x$ and $L_y$ (see \cite[p.~9]{stewart1998perturbation}). This leads to%
\begin{align*}
  \|s_2(x) - s_2(y)\|_p \leq \frac{\sqrt{n(n+1)}}{\ln 2} \sqrt{\|\sin \Theta(L^U_x,L^U_y)\|_F^2 + \|\sin \Theta(L^V_x,L^V_y)\|^2_F}.%
\end{align*}
Then  \emph{Wedin's theorem} (see \cite{cai2018rate} for a timely reference) states that the inequality $\delta > 0$ with $\delta$ as in \eqref{eq_def_delta} implies%
\begin{equation*}
  \sqrt{\|\sin \Theta(L^U_x,L^U_y)\|_F^2 + \|\sin \Theta(L^V_x,L^V_y)\|^2_F} \leq \frac{1}{\delta}\sqrt{\|R\|_F^2 + \|S\|_F^2},%
\end{equation*}
where $R$ and $S$ are matrices built from the singular value decompositions of $\zeta_x(p)$ and $\zeta_y(p)$ that satisfy $\max\{\|R\|_F,\|S\|_F\} \leq \|\zeta_x(p) - \zeta_y(p)\|_F$. Altogether,%
\begin{align*}
  \|s_2(x) - s_2(y)\|_p \leq \frac{\sqrt{2n(n+1)}}{\ln 2} \frac{1}{\delta}\|\zeta_x(p) - \zeta_y(p)\|_F.%
\end{align*}
Finally, we can estimate%
\begin{align*}
  \|\zeta_x(p) - \zeta_y(p)\|_F = \|p^{\frac{1}{2}}(A(x) - A(y))p^{-\frac{1}{2}}\|_F \leq \|p^{\frac{1}{2}}\|_F \|p^{-\frac{1}{2}}\|_F \|A(x) - A(y)\|_F.%
\end{align*}
The proof is complete.
\end{proof}

\begin{remark}
The assumption $\alpha_k(\zeta_y(p)) - \alpha_{k+1}(\zeta_x(p)) > 0$ will be satisfied if%
\begin{equation*}
  \alpha_k(\zeta_x(p)) - \alpha_{k+1}(\zeta_x(p)) > \|\zeta_x(p) - \zeta_y(p)\|,%
\end{equation*}
because the singular values are $1$-Lipschitz with respect to the matrix, implying%
\begin{equation*}
  \alpha_k(\zeta_y(p)) - \alpha_{k+1}(\zeta_x(p)) \geq -\|\zeta_x(p) - \zeta_y(p)\| + \alpha_k(\zeta_x(p)) - \alpha_{k+1}(\zeta_x(p)).%
\end{equation*}
\end{remark}

\section{Examples}
\label{sec:5}

In this section we will apply the theory  to   the H\'enon system with standard parameters $a=1.4$ and $b=0.3$, which is given by%
\begin{align*}
	x(t+1) &= 1.4 - x(t)^2 + 0.3y(t), \\
	y(t+1) &= x(t).%
\end{align*}The system has the two equilibria
$(x_+,x_+)$ and $(x_-,x_-)$, where $x_\pm=\frac{1}{2}(b-1\pm\sqrt{(b-1)^2+4a})$.
The quadrilateral $K$ with the following corners is a trapping region \cite{henon1976two}:%
\begin{align*}
	A &= (-1.862,1.96),\quad B = (1.848,0.6267) \\
	C &= (1.743,-0.6533),\quad D = (-1.484,-2.3333)%
\end{align*}
In particular, $K$ is a compact forward-invariant set.

We adapted the software from \cite{kawan2021software} for our computations; see also \cite{kawan2021subgradient}.  The maximization of $x\mapsto J_{k+s,x}(a,p)$ is performed on a $1000\times 1000$ grid and
is then refined on a grid of the same size around the maximizer found in the first step.  In all computations reported we used a fourth-degree polynomial $r_a$; we also performed computations with sixth- and eight-degree polynomials with very similar results.
W.l.o.g.~the constant term of the polynomial $r_a$ is always set equal to zero.
The exogenous step size  used was $t_k=16/k$ as in   \cite{kawan2021subgradient} and we always started with the metric $P=I$, i.e.~$p=I$ and $r_a=0$.
Let us first discuss the estimation of the dimension of the attractor.

\subsection{Dimension of the attractor}
We used exogenous step sizes and considered polynomials $r_a$ of degree 4.  With $k=1$ we experimented with different $s\in[0,1)$, the results are summarized in Table \ref{tab:dim}.  A few comments are in order.
For $s=0.430$ we did not obtain a negative value for $\max_{x\in K}\Sigma_{k+s,x}(P)$ in $100,000$ iterations. For $s\ge 0.435$ we obtained a negative value for $\max_{x\in K}\Sigma_{k+s,x}(P)$  and therefore the upper bound $\dim_L K \le 1.435$ on the
Lyapunov dimension.  For larger $s$ a negative value is obtained in fewer iterations and for $s\ge0.450$ no iterations are needed, i.e.~the metric $P=I$ is sufficient to prove the upper bound. 
The formula for the computed metric that gives the best upper bound $\dim_L K \le 1.435$ is shown in Table \ref{tab:polydim}.

\begin{table}[H]
\centering
Upper bounds on the Lyapunov dimension of the  H\'enon system on $K$\\
\vspace{2mm}
\begin{tabular}{|c|c|c|c|}
\hline
$\dim_L K $ &$s$ & First neg.~itr. & $\max_{x\in K}\Sigma_{k+s,x}(P) $\\ \hline
1.450      &    0.450      &    0 & -0.02278695484779664                        \\ \hline
1.445      &    0.445      &    641 & -0.005147904573907709            \\ \hline
1.440      &0.440      &1,061     &    -0.0001574233386151721                         \\ \hline
1.435      &0.435      &30,633     &    -0.0001650558073724702                           \\ \hline
$\star$       &0.430      & $\star$    &    $>0$                            \\ \hline
\end{tabular}
\vspace{0.2cm}
\caption{Results for the Lyapunov dimension estimate for various $s\in[0,1)$.  In all cases $k=1$ so $\dim_L K \le 1+s$. In the third column we write the first iteration where $\max_{x\in K}\Sigma_{k+s,x}(P)$ is negative and in the fourth column we give the value.  For $s=0.450$ no iterations are needed to obtain a negative value and for $s=0.430$ a negative value was not obtained in 100,000 iterations.
}
\label{tab:dim}
\end{table}

%

\begin{table}[H]
\centering
Metric $P= \rme^{r_a(x,y)}p$ for the  H\'enon system on $K$ that gives\\ the lowest upper bound $1.435$ on the Lyapunov dimension.
$$ p=\begin{pmatrix}
1.826992505777629 &0.0001682542238040804 \\
0.0001682542238040805 &0.5473476356577599
                 \end{pmatrix}
$$
\vspace{2mm}
\renewcommand{\arraystretch}{1.2}
\begin{tabular}{|c|c||c|c|}
\hline
\multicolumn{4}{|c|}{$r_a(x,y)$}\\ \hline
term &coefficient  & term &coefficient \\ \hline
$x$         & -0.7490964013161001 &$xy^2$      & 0.8389891922283548    \\ \hline
$y$         & 0.4240943498868161  &$y^3$       & -0.1068782681490099   \\ \hline
$x^2$       & 0.5582334651904182  &$x^4$       & 0.1074167044058492    \\ \hline
$xy$        & -0.2077703297215025 &$x^3y$      & 0.07468665833930994   \\ \hline
$y^2$       & -0.9333405988663968 &$x^2y^2$    & -0.0824147604660482   \\ \hline
$x^3$       & 0.09286451278055695 &$xy^3$      & -0.03892077756475409  \\ \hline
$x^2y$      & -0.1186578586959028 &$y^4$       & 0.4646176185293597    \\ \hline
\end{tabular}

\vspace{2mm}
\caption{
}
\label{tab:polydim}
\end{table}

For this example, the Lyapunov dimension of a bounded invariant set $K^*$ containing the equilibria $(x_-,x_-)$ and $(x_+,x_+)$ is $\dim_L K^*=1.4953$, see \cite[Thm.~3]{Leo} or \cite[Sect.~6.5.1]{Reit}. Note that our set $K$ does contain  $(x_+,x_+)$, but not $(x_-,x_-)$.
To apply our method to this case, we enlarged the set $K$ by changing the point $D$ to $(-2,-2.3333)$ to include the point $(x_-,x_-)$. Then  we obtained the upper bound $\dim_L K \le 1.5=1+s$ with $s=0.5$;  we did not obtain a negative value for  $\max_{x\in K}\Sigma_{k+s,x}(P)$ in 100,000 iterations with $s=0.49$.
Hence, our method found the upper bound $1.5$ on the actual value $1.4953$.

Other results in the literature include the numerical value of $1.220\pm 0.019$ for the fractal dimension of the attractor of the H\'enon system, see \cite{Martinez-Lopez}. Note that the Hausdorff dimension is bounded above by the fractal dimension, which in turn is bounded above by the Lyapunov dimension.

Back to the original set $K$ containing only $(x_+,x_+)$, we note that the Lyapunov dimension of the equilibrium $(x_+,x_+)$ is $1.3521$, which is thus a lower bound on the Lyapunov dimension of $K$. An upper bound, computed with our method is $1.435$.

\subsection{Restoration entropy}
\label{example:resent}
 In the first computation of the restoration entropy  we  used exogenous step sizes $t_k=16/k$ and no projection.  The results for the estimation of the restoration entropy are shown in Figure \ref{fig:exenoproj} on a log-log graph.  The initial upper bound with $P=I$ is $1.951140849266661$. The upper bound quickly increases and then settles down and converges.  The best upper bound is 1.313884120199142 obtained in iteration 99,654.  The formula for the computed metric that gives the best upper bound  is given in Table \ref{tab:exo000form}.

 \begin{figure}[H]
    \centering
    \includegraphics[width=0.8\linewidth]{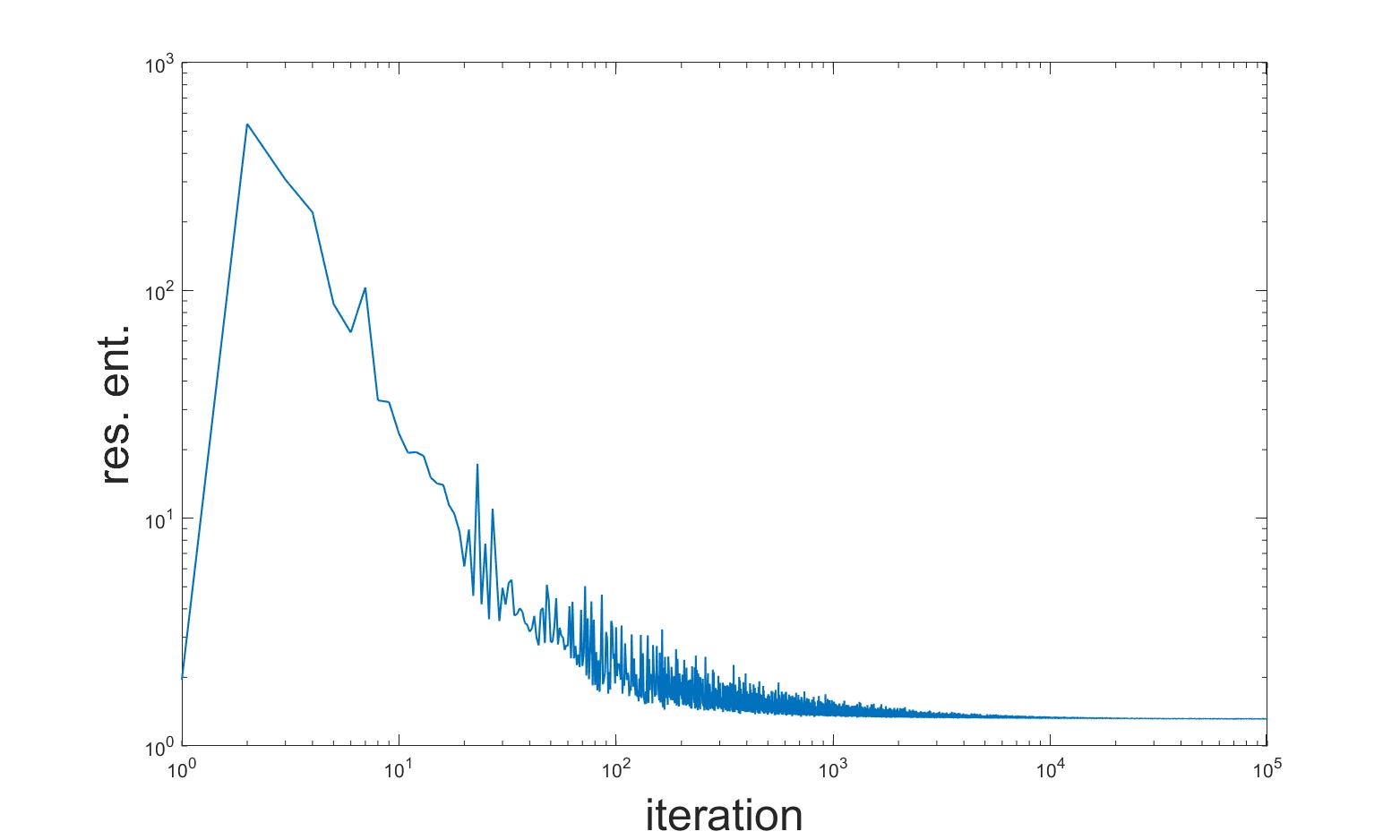}
    \caption{}
    \label{fig:exenoproj}
    \caption{Restoration entropy for the H\'enon system using exogenous step sizes $t_k=16/k$. The best upper bound is 1.313884120199142 obtained in iteration 99,654. }
\end{figure}

%
\begin{table}[H]
\centering
Metric $P= \rme^{r_a(x,y)}p$ for the  H\'enon system on $K$ that gives\\ the lowest upper bound 1.313884120199142 on the\\ restoration entropy (exogeneous step size).
$$ p=\begin{pmatrix}
1.717042057691368 &-0.01579983885287614\\
-0.01579983885287614 &0.5825423031576957
                 \end{pmatrix}
$$
\vspace{2mm}
\renewcommand{\arraystretch}{1.2}
\begin{tabular}{|c|c||c|c|}
\hline
\multicolumn{4}{|c|}{$r_a(x,y)$}\\ \hline
term &coefficient  & term &coefficient \\ \hline
$x$         & -0.03728625249859665 &$xy^2$      & 0.5850719733431917    \\ \hline
$y$         & 0.7173696947291418   &$y^3$       & 0.01827202349317083   \\ \hline
$x^2$       & 0.3852955420533915   &$x^4$       & 0.288243555432863     \\ \hline
$xy$        & -0.1800768847947239  &$x^3y$      & 0.07019941290373798   \\ \hline
$y^2$       & -0.5726642618756508  &$x^2y^2$    & -0.05631895479152903  \\ \hline
$x^3$       & -0.1943808570457239  &$xy^3$      & 0.022221575956417     \\ \hline
$x^2y$      & -0.3414860502248702  &$y^4$       & 0.3333574278923304    \\ \hline
\end{tabular}

\vspace{2mm}
\caption{}
\label{tab:exo000form}
\end{table}

Next we used fixed step sizes, $t_k=0.01$ and $t_k=0.001$, respectively.  The results are depicted in Figure \ref{fig:fs0001noproj}, step size $t_k=0.01$ in blue and step size $t_k=0.001$ in red.  The overshot at the beginning is much lower than when using exogenous step sizes and lower for the smaller step size.  Further, they do not converge to a value but oscillate in an interval, whose width is determined by the step size.  For step size $t_k=0.01$,  the best upper bound on the restoration entropy is  1.304961201612717, obtained in iteration 75,754. For step size $t_k=0.001$, the best upper is
1.306393103425693 obtained in iteration 99,279.  Note that although the best upper bounds are somewhat better than for the exogenous step sizes, they take longer to settle to reasonable estimates.  For example,  after the first 1,000 iterations the best upper bound is 1.341017992604001
for the exogenous step sizes but
1.394573310929959 and 1.573822400831881 for the fixed  step sizes $t_k=0.01$ and $t_k=0.001$ respectively.
The formulas for the computed metric that give the best upper bounds are presented in Table
\ref{fixed001} and
\ref{fixed0001} for step sizes $t_k=0.01$ and $t_k=0.001$, respectively.

\begin{figure}[H]
    \centering
    \includegraphics[width=0.8\linewidth]{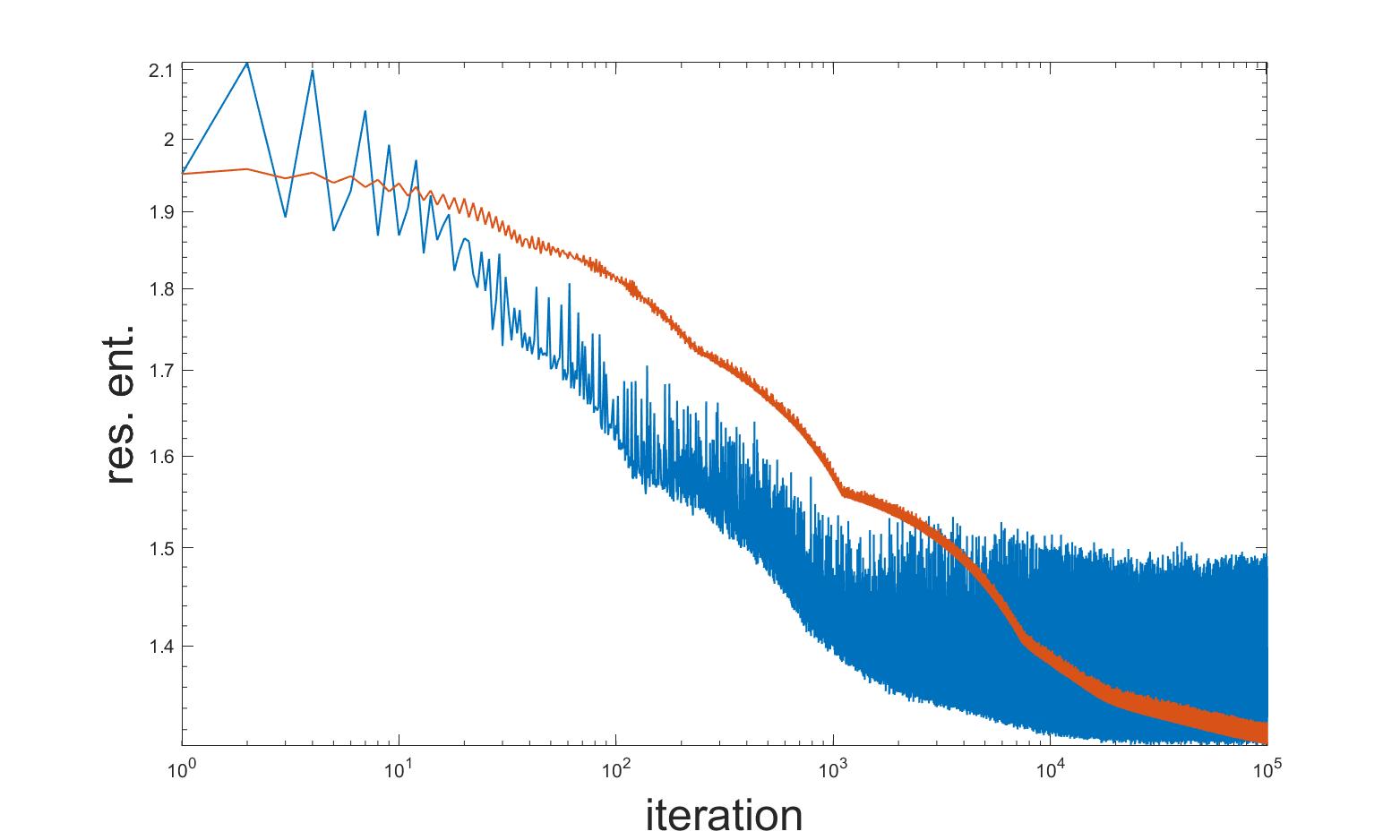}
    \caption{Restoration entropy for the H\'enon system using fixed step sizes  $t_k=0.01$ (blue) and $t_k=0.001$ (red), respectively. The best upper bound is 1.304961201612717 and 1.306393103425693, respectively. The bounds oscillate in an interval, the  size of which depends on the step size. }
    \label{fig:fs0001noproj}
\end{figure}
%

%
\begin{table}[H]
\centering
Metric $P= \rme^{r_a(x,y)}p$ for the  H\'enon system on $K$ that gives\\ the lowest upper bound 1.304961201612717 on the\\ restoration entropy (fixed step size  $t_k=0.01$)
$$ p=\begin{pmatrix}
1.825723762954203 &-2.675239129430998\cdot10^{-6}  \\
-2.675239129430998\cdot10^{-6} &0.5477279862255381
                 \end{pmatrix}
$$
\vspace{2mm}

\renewcommand{\arraystretch}{1.2}
\begin{tabular}{|c|c||c|c|}
\hline
\multicolumn{4}{|c|}{$r_a(x,y)$}\\ \hline
term &coefficient  & term &coefficient \\ \hline
$x$         & 0.5573870474721652   &$xy^2$      &  0.07724664642514245  \\ \hline
$y$         & 0.7077871346437571   &$y^3$       &  0.06313658395058322  \\ \hline
$x^2$       & 0.2275835960549717   &$x^4$       &  0.2599915790953258   \\ \hline
$xy$        & -0.149843153473276   &$x^3y$      &  0.06823486825177701  \\ \hline
$y^2$       & 0.1021644733650378   &$x^2y^2$    &  -0.03085636902009057 \\ \hline
$x^3$       & -0.1663567125415807  &$xy^3$      &  0.03118446283824575  \\ \hline
$x^2y$      & -0.3651316623899711  &$y^4$       &  0.05509114432195542  \\ \hline
\end{tabular}

\vspace{2mm}
\caption{
}
\label{fixed001}
\end{table}
\begin{table}[H]
\centering
Metric $P= \rme^{r_a(x,y)}p$ for the  H\'enon system on $K$ that gives\\ the lowest upper bound 1.306393103425693 on the\\ restoration entropy (fixed step size  $t_k=0.001$)
$$ p=\begin{pmatrix}
1.598582176519417 &-0.0340879261290964      \\
-0.03408792612909639 &0.6262812142912478
                 \end{pmatrix}
$$
\vspace{2mm}

\renewcommand{\arraystretch}{1.2}
\begin{tabular}{|c|c||c|c|}
\hline
\multicolumn{4}{|c|}{$r_a(x,y)$}\\ \hline
term &coefficient  & term &coefficient \\ \hline
$x$         & 0.4301700400628248  &$xy^2$      &  0.1991442773311169   \\ \hline
$y$         & 0.5742005669502973  &$y^3$       &  0.09368300218613741  \\ \hline
$x^2$       & 0.2138324226849934  &$x^4$       &  0.2663176394736673   \\ \hline
$xy$        & -0.2027704610258738 &$x^3y$      &  0.09374967991423779  \\ \hline
$y^2$       & 0.04317190264405415 &$x^2y^2$    &  0.007144757712718732 \\ \hline
$x^3$       & -0.1495907154564866 &$xy^3$      &  0.05518248936958972  \\ \hline
$x^2y$      & -0.2981438582431617 &$y^4$       &  0.0808752884012343   \\ \hline
\end{tabular}

\vspace{2mm}
\caption{
}
\label{fixed0001}
\end{table}

In Table \ref{tab:norms} we list the eigenvalues of $p$ and the norm of the vector $a$ defining the polynomial $r_a$ in the metric that delivers the lowest upper bound on the restoration entropy for the step sizes we used.

\begin{table}[h!]
\begin{tabular}{|c|c|c|c|}
  \hline
   &norm $a$  & min eigenvalue $p$ & max eigenvalue $p$ \\ \hline
  exogenous & 1.313270369387425 & 0.582322306187126 & 1.717262054661938 \\
  $t_k=0.01$ & 1.069807454767147 & 0.547727986219938 & 1.825723762959803 \\
  $t_k=0.001$ & 0.923314506269372 & 0.625087590080113 & 1.599775800730552 \\
  \hline
\end{tabular}
\vspace{0.2cm}

\caption{The eigenvalues of $p$ and norm of the vector $a$ defining the polynomial $r_a$ for the optimal metrics obtained with the different step sizes.}  \label{tab:norms}
\end{table}

From the table we deduct that a projection of  $a$ to the ball with radius $1.5$ in $\R^{15}$ and of $p$ to the interval $[0.5I,2I]$ should not be limiting.  We repeat the computations above using the projected sub-gradient algorithm.
For the exogenous step sizes the results were slightly different, see Figure \ref{fig:compexoprj}, and the bound on the restoration entropy was lower, i.e.~1.305419886141552 obtained in iteration 99,456.
The values for the best metric are displayed in Table \ref{tab:project}.

For the fixed step sizes they were so similar that we abstain from plotting them.  The difference in the bound on the restoration entropy was in the fourth decimal place.  For the step size $t_k=0.001$ we also tried projecting  $p$ on
$[0.61I,1.61I]$ and $a$ to the ball with radius $0.95$, i.e.~very close the the values we had found for the optimal metric.  The results were very similar to the results without projection with no notable improvement.
Note that the results with fixed step size are  what one would expect from the error analysis results for the projected subgradient algorithm with fixed step size, cf.~\eqref{eq:complexityConstant}:  the estimates oscillate in an interval and do not converge.

\begin{figure}[H]
    \centering
    \includegraphics[width=0.8\linewidth]{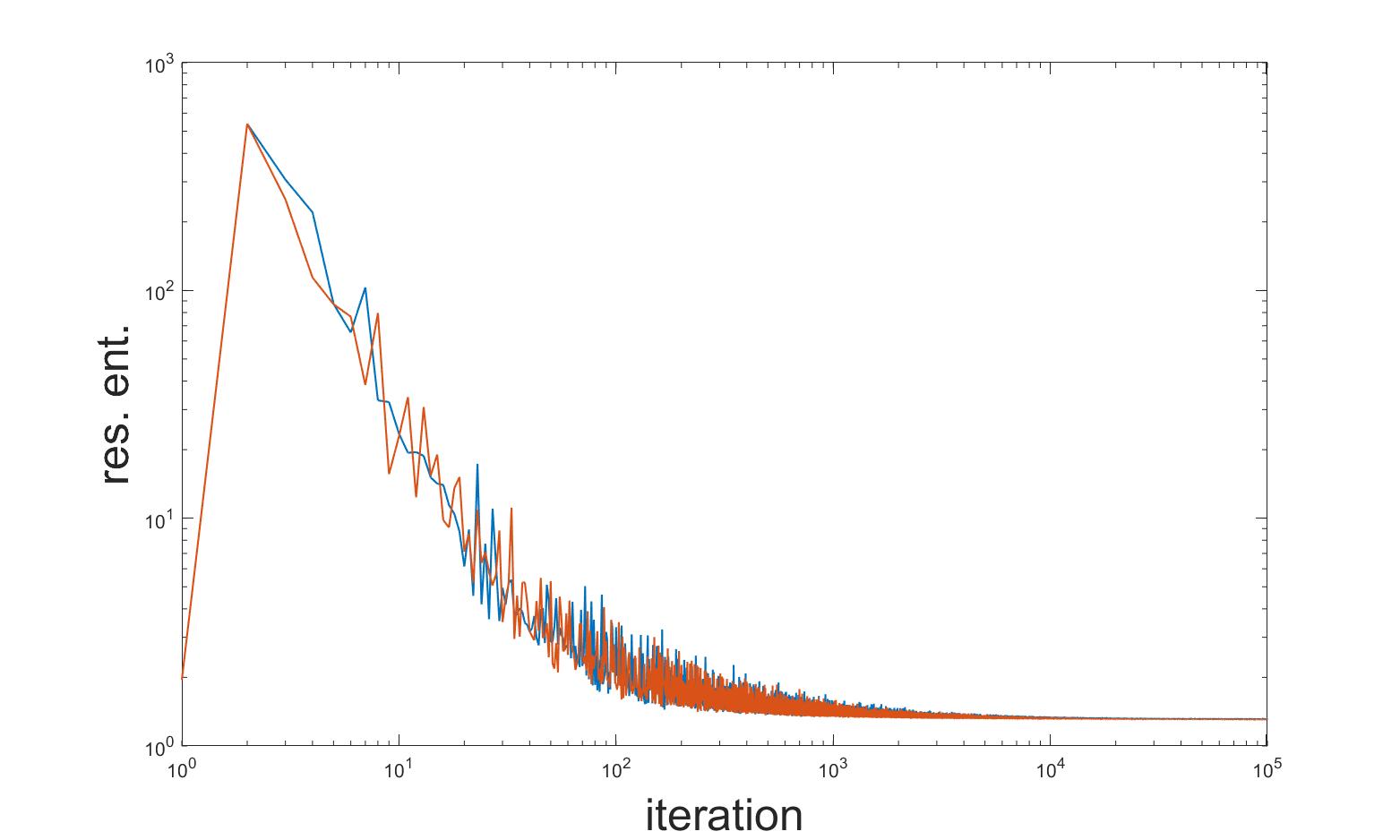}
    \caption{Restoration entropy for the H\'enon system using exogenous step sizes $t_k=16/k$, with projection  of  $a$ to the ball with radius $1.5$ in $\R^{15}$ and of $p$ to the interval $[0.5I,2I]$ (red) and without projection (blue, see also Fig.~\ref{fig:exenoproj}). The best upper bound of the projected version is ~1.305419886141552 obtained in iteration 99,456, while the best result for the one without projection  was 1.313884120199142 obtained in iteration 99,654.}
    \label{fig:compexoprj}
\end{figure}

%

\begin{table}[H]
\centering
Metric $P= \rme^{r_a(x,y)}p$ for the  H\'enon system on $K$ that gives\\ the lowest upper bound 1.305419886141552 on the restoration entropy\\ (exogeneous step size with projection to ball of radius 1.5 and $[0.5I,2I]$).
$$ p=\begin{pmatrix}
1.674120023285701 &-0.02266875912007223 \\
-0.02266875912007223 &0.5976356884178141
                 \end{pmatrix}
$$
\vspace{2mm}

 \renewcommand{\arraystretch}{1.2}
 \begin{tabular}{|c|c||c|c|}
 \hline
 \multicolumn{4}{|c|}{$r_a(x,y)$}\\ \hline
 term &coefficient  & term &coefficient \\ \hline
 $x$         & 0.3113339501091084    &$xy^2$      &  0.2845946891080638   \\ \hline
 $y$         & 0.7114011283722037    &$y^3$       &  0.03537685809298098  \\ \hline
 $x^2$       & 0.2642178218204402    &$x^4$       &  0.2828190161018292   \\ \hline
 $xy$        & -0.1555315664585948   &$x^3y$      &  0.06467714057792809  \\ \hline
 $y^2$       & -0.2072950713728302   &$x^2y^2$    &  -0.04965314808703013 \\ \hline
 $x^3$       & -0.1894657057265681   &$xy^3$      &  0.02054268087684425  \\ \hline
 $x^2y$      & -0.3660416537330809   &$y^4$       &  0.1771757807926334   \\ \hline
 \end{tabular}

\vspace{2mm}
\caption{
}
\label{tab:project}
\end{table}

When the projection is to a subset of $\SC_2^+\times \R^{15}$, where an optimal metric is not located, the subgradient method converges to a higher value of the bound on the restoration entropy; see Figure \ref{fig:badpro} where the projection is to
a ball with radius $0.5$ for $a$ and $[1.5I,2I]$ for $p$ and the upper bound converges to 1.387955604316810.  The metric is given in Table \ref{badprojectionX1}.

\begin{figure}[H]
    \centering
    \includegraphics[width=0.8\linewidth]{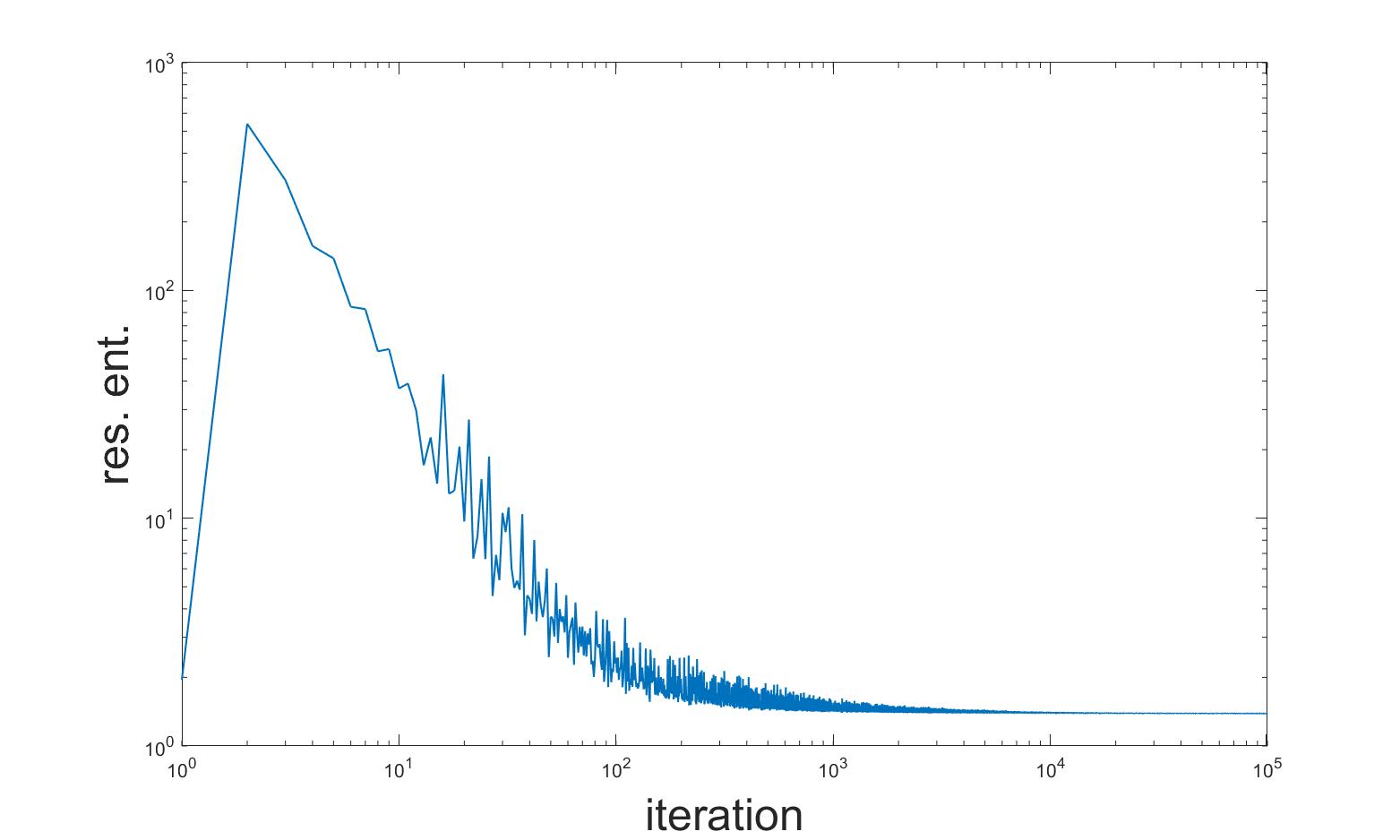}
    \caption{Restoration entropy for the H\'enon system using exegenous step sizes $t_k=16/k$ and  projection  of  $a$ to the ball with radius $0.5$ in $\R^{15}$ and of $p$ to the interval $[1.5I,2I]$. The best upper bound on
     the restoration entropy was 1.387955604316810. The optimal metric is not in the set projected on. }
    \label{fig:badpro}
\end{figure}
%
\begin{table}[H]
\centering

Metric $P= \rme^{r_a(x,y)}p$ for the  H\'enon system on $K$ that gives the lowest upper bound 1.387955604316810 on the restoration entropy (exogenous step size and projection to a ball of radius 0.5 and $[1.5I,2I]$; the optimal metric is not in this set). 
$$ p=\begin{pmatrix}
1.960938619846507 &-0.1342041312602802 \\
-0.1342041312602802 &1.539064363507252
                 \end{pmatrix}
$$
\vspace{2mm}

\renewcommand{\arraystretch}{1.2}
\begin{tabular}{|c|c||c|c|}
\hline
\multicolumn{4}{|c|}{$r_a(x,y)$}\\ \hline
term &coefficient  & term &coefficient \\ \hline
$x$         & 0.1032680736194371      &$xy^2$      &  0.1575577522285691   \\ \hline
$y$         & 0.1810915872324425      &$y^3$       &  0.1919490860856328   \\ \hline
$x^2$       & -0.01629057454520433    &$x^4$       &  0.2630121488080179   \\ \hline
$xy$        & -0.1664244112423047     &$x^3y$      &  0.07944553907063254  \\ \hline
$y^2$       & 0.02202230639431921     &$x^2y^2$    &  0.08610480234652296  \\ \hline
$x^3$       & 0.002088602994479542    &$xy^3$      &  0.1177457076550595   \\ \hline
$x^2y$      & -0.1377282628515208     &$y^4$       &  0.02732136529820107  \\ \hline
\end{tabular}

\vspace{2mm}
\caption{
}
\label{badprojectionX1}
\end{table}

Finally, we used Polyak step sizes.  From our results we deemed it sensible to set $f^*:=1.3$  in \eqref{StepsizePolyak}  and use the projection $p$ to the interval $[0.5I,2I]$ and $a$ to the ball with radius $1.5$ in $\R^{15}$.
 We tried $\alpha$ equal to $0.5$, $1$, and $1.5$ times $\frac{\tanh\left(\hat{\kappa}D \right)}{\hat{\kappa }D}$.  The best results were obtained with $1.5$: the upper bound 1.300531470950855 on the restoration entropy in iteration
 98,861. The metric is displayed in Table \ref{tab:Polyak}. This was the best result we obtained in our computations.

 \begin{figure}[H]
    \centering
    \includegraphics[width=0.8\linewidth]{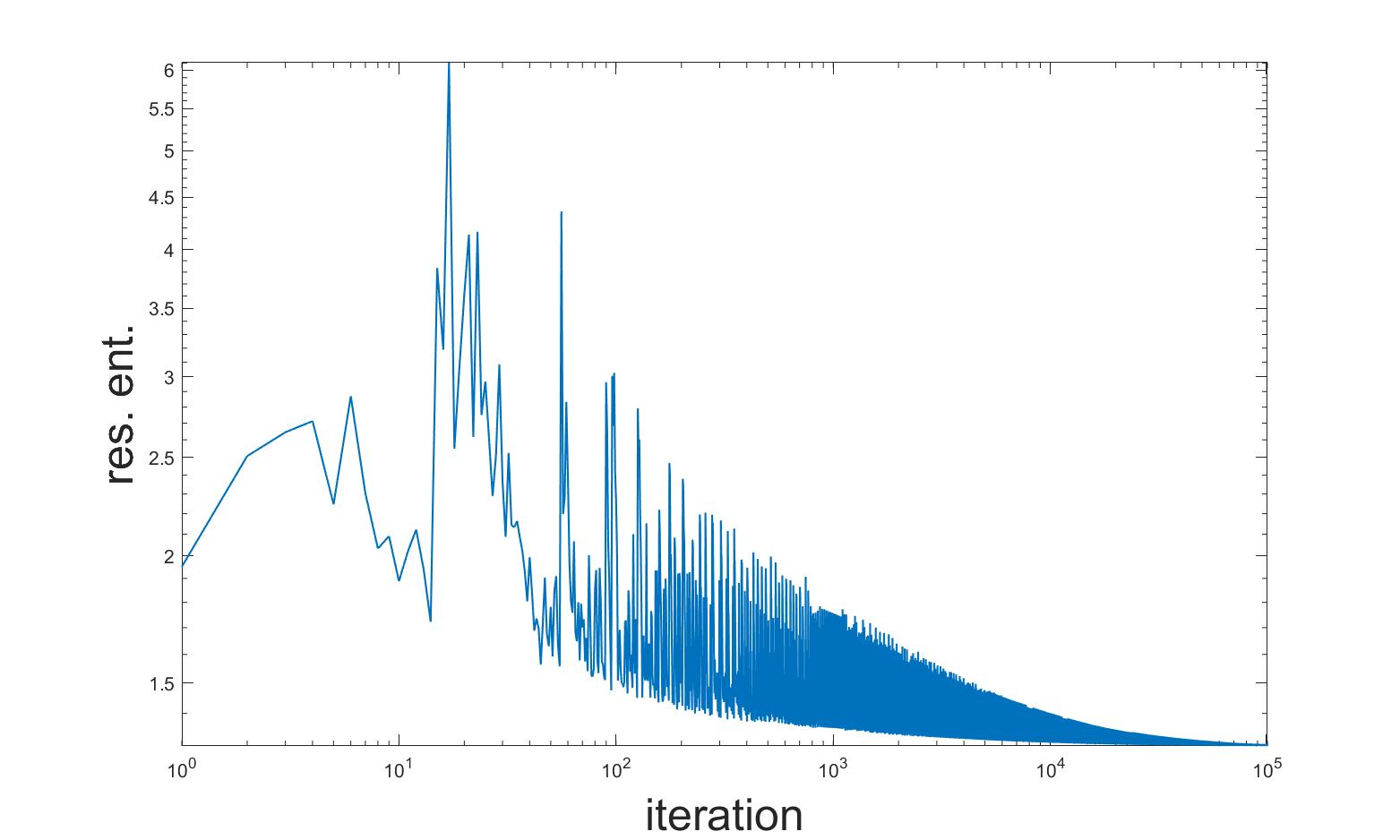}
    \caption{Restoration entropy for the H\'enon system using Polyak step sizes with $\alpha=1.5\frac{\tanh\left(\hat{\kappa}D \right)}{\hat{\kappa }D}$ and using projection  of  $a$ to the ball with radius $1.5$ in $\R^{15}$ and of $p$ to the interval $[0.5I,2I]$. The best upper bound on
     the restoration entropy was 1.300531470950855 obtained in iteration 98,861.}
    \label{fig:polyak}
\end{figure}

%

\begin{table}[H]
\centering
Metric $P= \rme^{r_a(x,y)}p$ for the  H\'enon system on $K$ that gives\\ the lowest upper bound 1.300531470950855 on the restoration entropy\\ (Polyak step size with $f^*=1.3$, projection to $[0.5I,2I]$ and 
 $\alpha=1.5\frac{\tanh\left(\hat{\kappa}D \right)}{\hat{\kappa }D}$.).
$$ p=\begin{pmatrix}
1.738181852184494 &-0.01314292044388416 \\
-0.01314292044388416 &0.5754131739079722
                 \end{pmatrix}
$$
\vspace{2mm}

\renewcommand{\arraystretch}{1.2}
\begin{tabular}{|c|c||c|c|}
\hline
\multicolumn{4}{|c|}{$r_a(x,y)$}\\ \hline
term &coefficient  & term &coefficient \\ \hline
$x$         & 0.5043455401581707      &$xy^2$      &  0.1066076072540229    \\ \hline
$y$         & 0.700029597798829       &$y^3$       &  0.06332507211373492   \\ \hline
$x^2$       & 0.2393485122804599      &$x^4$       &  0.2636011379394431    \\ \hline
$xy$        & -0.1604171490173032     &$x^3y$      &  0.06845945107788326   \\ \hline
$y^2$       & 0.05239219536145903     &$x^2y^2$    &  -0.02920581871325886  \\ \hline
$x^3$       & -0.1702108324691254     &$xy^3$      &  0.03270559792757753   \\ \hline
$x^2y$      & -0.3700063614833601     &$y^4$       &  0.06889085991787346   \\ \hline
\end{tabular}

\vspace{2mm}
\caption{
}
\label{tab:Polyak}
\end{table}

\section{Conclusions}

In this paper we have overcome several shortcomings of the subgradient algorithm with applications to derive upper bounds on the dimension of attractors and the entropy of systems. In particular, we have introduced a projected subgradient method on a compact subset of 
the domain which guarantees that for this restricted problem a minimizer exists. We demonstrated the applicability of our novel algorithm for upper bounding the Lyapunov dimension and the restoration entropy of the H\'enon map.

\appendix%

\section{Sectional curvature of the space of positive matrices}

In this section, we compute the tightest lower bound on the sectional curvature of the Riemannian manifold $\SC^+_n$. Our analysis is based on the fact that $\SC^+_n$ has the structure of a \emph{Riemannian symmetric space}. We refer to \cite[ch.~4]{helgason1979differential} for the theory of symmetric spaces. We use the notation $\SO(n) = \{ A \in \GL(n,\R) : AA\trn = I,\ \det(A) = 1\}$ for the special orthogonal group.%

By \cite[Thm.~3.3]{helgason1979differential}, every symmetric space $M$ is a homogeneous space, i.e.~there is a Lie group $G$ acting transitively on $M$ such that $M$ is diffeomorphic to the Lie group quotient $G/K$ with $K = \{g \in G: gp=p\}$ being the isotropy group of a base point $p \in M$ (that can be chosen arbitrarily). In fact, one can take $G$ to be the isometry group of $M$.%

To compute the curvature of a homogeneous space, it obviously suffices to compute the curvature at one point. Since the identity matrix $I \in \SC^+_n$ is a canonical choice, we only look at the sectional curvature at $p = I$.%

We will use the following result, see \cite[Prop.~3.4]{helgason1979differential}.%

\begin{proposition}
Let $G$ be a connected Lie group, $\sigma:G \rightarrow G$ an involutive automorphism and $G^{\sigma} := \{g \in G : \sigma(g) = g \}$ the associated group of fixed points. Let $G^{\sigma}_0$ be the connected component of $G^{\sigma}$ which contains the identity and assume that $G^{\sigma}_0$ is compact. Then, for any compact subgroup $K \subset G$ with $G^{\sigma}_0 \subset K \subset G^{\sigma}$, the homogeneous space $G/K$, equipped with a $G$-invariant metric, is a symmetric space.
\end{proposition}

For a better understanding of the proposition, we recall that a $G$-invariant metric on $G/K$ is a Riemannian metric in which all of the maps $xK \mapsto gxK$, $g \in G$, are isometries. To use the proposition for our purposes, let%
\begin{equation*}
  G := \GL^+(n,\R) = \{ g \in \GL(n,\R) : \det(g) > 0 \},%
\end{equation*}
which is clearly a connected Lie group. We define%
\begin{equation*}
  \sigma:G \rightarrow G,\quad g \mapsto (g\trn)^{-1}.%
\end{equation*}
This map is an involutive automorphism, since $\sigma(g_1g_2) = (g_2\trn g_1\trn)^{-1} = (g_1\trn)^{-1}(g_2\trn)^{-1} = \sigma(g_1)\sigma(g_2)$ and $\sigma^2(g) = g$. We compute%
\begin{equation*}
  G^{\sigma} = \{ g \in G: (g\trn)^{-1} = g \} = \{ g \in G : gg\trn = I \} = \SO(n).%
\end{equation*}
Clearly, $G^{\sigma}$ is connected, so $G^{\sigma} = G^{\sigma}_0$ and we can only choose $K = \SO(n)$. It follows that $G/K = \GL^+(n,\R)/\SO(n)$ is a symmetric space when equipped with a $G$-invariant metric. We can identify the space $G/K$ with $\SC^+_n$ via the mapping%
\begin{equation*}
  \varphi:G/K \rightarrow \SC^+_n,\quad gK \mapsto gg\trn.%
\end{equation*}
Indeed, any matrix of the form $gg\trn$ with $g \in G$ is symmetric and positive definite. Conversely, every $p \in \SC^+_n$ can be written as $p = \varphi(p^{\frac{1}{2}})$. We have the equivalences%
\begin{align*}
   g_1K = g_2K \quad &\Leftrightarrow \quad g_2^{-1}g_1 \in \SO(n) \\
               \quad &\Leftrightarrow \quad g_2^{-1}g_1 = ((g_2^{-1}g_1)\trn)^{-1} = g_2\trn (g_1\trn)^{-1} \\
							 \quad &\Leftrightarrow \quad g_1g_1\trn = g_2g_2\trn.%
\end{align*}
We have thus shown that $\varphi$ is well-defined and bijective. To do the curvature computations in $G/K$ instead of $\SC^+_n$, the Riemannian metric on $G/K$ needs to be defined in such a way that $\varphi$ is an isometry. First, observe that the pullback of the metric that we defined on $\SC^+_n$ to $G/K$ via $\varphi$ yields a $G$-invariant metric on $G/K$, since the left translations $xK \mapsto gxK$ correspond to the maps $p \mapsto gpg\trn$ on $\SC^+_n$ under the given identification and these maps preserve the trace metric.%

The tangent space of $G/K$ at the base point can be identified with the space of symmetric $n \tm n$ matrices, since $T_IG = \R^{n \tm n}$ and $T_IK = \{ X \in \R^{n \tm n} : X = -X\trn \}$. To compute the derivative of $\varphi$ at the base point, consider a smooth curve $\gamma:\R \rightarrow G/K$, $\gamma(t) = \tilde{\gamma}(t)K$ with $\tilde{\gamma}:\R \rightarrow G$ smooth and $\tilde{\gamma}(0) = I$, $v := \dot{\tilde{\gamma}}(0) \in \SC_n$. Then%
\begin{align*}
  \rmD\varphi(K)v = \frac{\rmd}{\rmd t}\Bigl|_{t=0} \varphi(\gamma(t)) = \frac{\rmd}{\rmd t}\Bigl|_{t=0} \tilde{\gamma}(t)\tilde{\gamma}(t)\trn = v + v\trn = 2v.%
\end{align*}
Since we consider the pullback of the trace metric to $G/K$ via $\varphi$, this map needs to be an isometry. That is, the inner product on $T_K(G/K) = \SC_n$ has to be defined by%
\begin{equation}\label{eq_lie_metric}
  \langle X,Y \rangle = 4\tr(XY).%
\end{equation}

According to \cite[Thm.~4.2]{helgason1979differential}, the Riemannian curvature tensor of $G/K$ at the point $p_0 = K$ can be computed as%
\begin{equation*}
  R(X^*,Y^*)Z^*(p_0) = -[[X,Y],Z]^*(p_0) \mbox{\quad for all\ } X,Y,Z \in \SC_n.%
\end{equation*}
Here, $[X,Y] = XY - YX$ is the commutator of matrices and%
\begin{equation*}
  X^*(p) := \frac{\rmd}{\rmd t}\Bigl|_{t=0} \exp(tX) \cdot p,%
\end{equation*}
where $\cdot$ denotes the canonical action of $G$ on $G/K$ and $\exp$ is the Lie group exponential map. In our case,%
\begin{equation*}
  X^*(p_0) = \frac{\rmd}{\rmd t}\Bigl|_{t=0} \pi(\exp(tX)),%
\end{equation*}
where $\exp(tX) = \sum_{k=0}^{\infty}\frac{1}{k!}t^kX^k$ and $\pi:G \rightarrow G/K$, $\pi(g) = gK$, is the canonical projection. It follows that%
\begin{equation*}
  X^*(p_0) = \rmD\pi(I) X.%
\end{equation*}
Since $\rmD\pi(I)$ is the identity on $\SC_n$, it follows that $X^*(p_0) = X$ and thus%
\begin{equation*}
  R(X^*,Y^*)Z^*(p_0) = -[[X,Y],Z] \mbox{\quad for all\ } X,Y,Z \in \SC_n.%
\end{equation*}
Using \eqref{eq_lie_metric}, the sectional curvature for a plane spanned by two linearly independent $X,Y \in \SC_n$ can then be computed as%
\begin{align*}
  K(X,Y) &= \frac{\langle R(X,Y)Y,X \rangle}{\|X\|^2\|Y\|^2 - \langle X,Y \rangle^2} = \frac{4 \tr( -[[X,Y],Y] X )}{ 4\tr(X^2) 4\tr(Y^2) - 16 \tr(XY)^2} \\
				 &= -\frac{1}{4}\frac{\tr( XY^2X - YXYX - YXYX + Y^2X^2 )}{\tr(X^2)\tr(Y^2) - \tr(XY)^2} \\
				 &= -\frac{1}{4}\frac{\tr( 2Y^2X^2 - 2 (XY)^2 )}{\tr(X^2)\tr(Y^2) - \tr(XY)^2} = -\frac{1}{2}\frac{\tr(X^2Y^2) - \tr((XY)^2)}{\tr(X^2)\tr(Y^2) - \tr(XY)^2}.%
\end{align*}
Here, we used that $\tr(AB) = \tr(BA)$ for arbitrary $A,B$ and the linearity of $\tr$.%

Since the sectional curvature of a plane does not depend on the choice of the basis vectors, it suffices to consider orthonormal $X$ and $Y$. Hence, to find the smallest possible curvature, we have so solve the minimization problem%
\begin{align*}
  \min [\tr((XY)^2) - \tr(X^2 Y^2)] \mbox{\quad s.t.\ } \tr(X^2) = \tr(Y^2) = 1,\ \tr(XY) = 0.%
\end{align*}
Putting $A := XY$, we obtain%
\begin{align*}
  \tr((XY)^2) - \tr(X^2 Y^2) &= \tr(A^2) - \tr(XYYX) = \tr(A^2) - \tr(AA\trn) \\
	&= \tr( A(A - A\trn) ) = \langle A,A\trn - A \rangle_F.%
\end{align*}
We decompose $A$ into its symmetric and its skew-symmetric part: $A = A_+ + A_-$ with $A_+ = \frac{1}{2}(A + A\trn)$ and $A_- = \frac{1}{2}(A - A\trn)$. Since $\langle A_+,A_- \rangle_F = 0$, we obtain%
\begin{equation*}
  \tr((XY)^2) - \tr(X^2 Y^2) = \frac{1}{2} \langle A - A\trn,A\trn - A \rangle_F = -\frac{1}{2}\|A - A\trn\|^2_F.%
\end{equation*}
Using the main result of \cite{bottcher2008frobenius}, we can estimate%
\begin{equation*}
  \|A - A\trn\|^2_F = \| XY - YX \|^2_F = \|[X,Y]\|_F^2 \leq (\sqrt{2} \|X\|_F \|Y\|_F)^2 = 2.%
\end{equation*}
Hence, we obtain that $K(X,Y) \geq - \frac{1}{2}$. The results of \cite[Sec.~4]{bottcher2008frobenius} show that this bound is tight. Using \cite[Lem.~A.1]{kawan2021subgradient}, we obtain the following result.%

\begin{proposition}\label{prop_sec_curv_lb}
For each $n \in \N$, the smallest possible lower bound on the sectional curvature of $\SC^+_n$ is $-\frac{1}{2}$ and the same is true for any space of the form $\R^N \tm \SC^+_n$ with $N \in \N$.
\end{proposition}

\bibliographystyle{abbrv}
\bibliography{ProjSubg}

\end{document}